\documentclass[reqno,11pt]{amsart}
\usepackage{lineno,amsmath,amsthm,amssymb,graphics,epsfig,mathrsfs, eufrak}
\usepackage{enumerate,geometry,etoolbox, scalerel, stackengine, relsize}
\usepackage{amscd,latexsym,url,upgreek,mathtools,color}
\usepackage{mathrsfs}
\usepackage{hyperref}
\usepackage[english]{babel}

\stackMath
\newcommand\reallywidehat[1]{%
\savestack{\tmpbox}{\stretchto{%
  \scaleto{%
    \scalerel*[\widthof{\ensuremath{#1}}]{\kern-.6pt\bigwedge\kern-.6pt}%
    {\rule[-\textheight/2]{1ex}{\textheight}}
  }{\textheight}%
}{0.5ex}}%
\stackon[1pt]{#1}{\tmpbox}%
}

\newtheorem{theorem}{Theorem}[section]
\newtheorem{lemma}[theorem]{Lemma}

\newtheorem{proposition}[theorem]{Proposition}
\newtheorem{example}[theorem]{Example}

\theoremstyle{definition}

\newtheorem{definition}[theorem]{Definition}

\def\C{\mathbb C}

\def\R{\mathbb R}

\def\N{\mathbb N}

\def\Z{\mathbb Z}

\def\D{\mathbb D}

\def\bfJ{{\boldsymbol J}}
\def\bfS{{\boldsymbol S}}
\def\bfT{{\boldsymbol T}}

\def\bfr{{\boldsymbol r}}

\def\bfa{{\boldsymbol a}}

\def\bfN{{\boldsymbol N}}

\def\bfa{{\boldsymbol a}}

\def\bfh{{\boldsymbol h}}

\def\bfw{{\boldsymbol w}}

\def\bfz{{\boldsymbol z}}

\def\bfalpha{{\boldsymbol \alpha}}

\def\bfphi{{\boldsymbol \phi}}

\def\bfeta{{\boldsymbol \eta}}

\def\bfnu{{\boldsymbol \nu}}
\def\bfomega{{\boldsymbol \omega}}
\def\bfrho{{\boldsymbol \rho}}
\def\bftau{{\boldsymbol \tau}}
\def\bfkappa{{\boldsymbol \kappa}}

\def\({{\rm (}}
\def\){{\rm )}}

\begin{document}

\numberwithin{equation}{section}

\title{Analyticity, superoscillations and supershifts in several variables}

\author[F. Colombo]{F. Colombo}
\address{(FC) Politecnico di
Milano\\Dipartimento di Matematica\\Via E. Bonardi, 9\\20133 Milano\\Italy}
\email{fabrizio.colombo@polimi.it}

\author[I. Sabadini]{I. Sabadini}
\address{(IS) Politecnico di
Milano\\Dipartimento di Matematica\\Via E. Bonardi, 9\\20133 Milano\\Italy}
\email{irene.sabadini@polimi.it}

\author[D. C. Struppa]{D. C. Struppa}
\address{(DCS) The Donald Bren Presidential Chair in Mathematics\\ Chapman University, Orange, CA 92866 \\ USA}
\email{struppa@chapman.edu}

\author[A. Yger]{A. Yger}
\address{(AY) IMB, Universit\'e de Bordeaux, 33405, Talence, France}
\email{yger@math.u-bordeaux.fr}
\date{\today}


\begin{abstract}

Superoscillations have roots in various scientific disciplines, including optics, signal processing, radar theory, and quantum mechanics. This intriguing mathematical phenomenon permits specific functions to oscillate at a rate surpassing their highest Fourier component.
A more encompassing concept, supershifts, extends the idea of superoscillations to functions that are not sum of exponential functions. This broader notion is linked to Bernstein and Lagrange approximation of analytic functions in $\mathbb{C}^n$. Recent advancements in the theory of superoscillations and supershifts in one variable have focused on their time evolution. This paper takes a step further by expanding the notion of supershifts to include the case of several variables. We provide specific examples related to harmonic analysis where the variables vary in multi-dimensional frequency (space, or scale) domains.

\end{abstract}
\maketitle



\section{Introduction}
Superoscillations, both a physical phenomenon and a mathematical concept, refer to functions or
sequences exhibiting the intriguing ability to oscillate at a rate surpassing what their highest
Fourier component would conventionally allow.
This unique property holds significant promise for scientific and technological progress, see \cite{Be19,QS20,OPTICS,SPGR,KM1,sodakemp}. These functions manifest in various contexts, notably in quantum mechanics \cite{aav,abook}, where they emerge from weak measurements. Intriguing questions arise regarding the evolution of these functions governed by Quantum Fields Equations.

Recently, there has been widespread research on the evolution of superoscillations as initial conditions for Schr\"odinger equations \cite{ABCS22,JDE,ACSST17A,b4,b5,peter,Pozzi,SPIN} leading to the emergence of new inquiries and questions.

This development serves as a vital link between the mathematical understanding and practical
applications of these field of studies, offering innovative insights across diverse fields, not only in
quantum mechanics. Our primary objective in this paper is to delve into the mathematical theory of these functions in several variables where also recently there has been a significant progress, with a broader focus on supershifts. The concept of supershift generalizes that of
superoscillations, which is a specific case.

Consider a real number $a>1$: the prototypical superoscillatory sequence, initially introduced in
the context of weak values theory, consists of the complex-valued functions ${F_N(x,a)}$ defined on
the real line $\mathbb{R}$ by
\begin{equation}\label{FNEXP}
F_N(x,a)=\sum_{\nu=0}^N{N\choose \nu}\left(\frac{1-a}{2}\right)^{\nu}\left(\frac{1+a}{2}\right)^{N-\nu} e^{i(1-2\nu/N){x}}
\end{equation}
where ${N\choose \nu}$ denotes the binomial coefficients.
The initial observation is that by fixing $x \in \mathbb{R}$ and allowing $N$ to approach infinity, we easily deduce that $F_N(x,a)$ converges to $e^{iax}$ uniformly for $x$ belonging to compact sets in $\mathbb{R}$.
The class of superoscillating functions can be broadened by extending both the set of coefficients ${N\choose \nu}\left(\frac{1-a}{2}\right)^{\nu}\left(\frac{1+a}{2}\right)^{N-\nu}$ and the sequence of frequencies $h(N,\nu)=1-2\nu/N$ ensuring that they are bounded by 1. A more extensive collection of superoscillating functions, as outlined in \cite{ACSSST21} and further studies in \cite{BOREL,cssy23,IRREGULAR-SAMP}, is defined under the condition that the points $h(N,\nu)$ for $\nu=0,...,N$ are distinct and so the functions have an explicit form given by
$$
f_N(x)=\sum_{\nu=0}^N\Big(\prod_{\nu'=0,\  \nu'\not=\nu}^N\Big(\frac{a- h(N,\nu')}{h(N,\nu)-h(N,\nu')}\Big)\Big)\ e^{ih(N,\nu) x},\ \ \ x\in \mathbb{R}.
$$

Based on the previously outlined families of superoscillatory functions, we can establish the definitions for generalized Fourier series and superoscillating functions.

We define a \textit{generalized Fourier sequence} as a sequence of the form
\begin{equation}\label{basic_sequence}
f_N(x):= \sum_{\nu=0}^N X_\nu(N,a)e^{ih(N,\nu)x},\ \ \ N\in \mathbb{N},\ \ \ x\in \mathbb{R},
\end{equation}
where $a\in\mathbb R$, $X_\nu(N,a)$ and $h(N,\nu)$
are complex and real-valued functions of the variables $N,a$ and $N$, respectively.
A generalized Fourier sequence, as expressed in (\ref{basic_sequence}),
is termed \textit{a superoscillating sequence} if
$\sup_{\nu,N}|h(N,\nu)|\leq 1$ and
there exists a compact subset of $\mathbb R$,
referred to as \textit{a superoscillation set},
on which $f_N(x)$ converges uniformly to $e^{ig(a)x}$,
where $g$ is a continuous real-valued function with $|g(a)|>1$.

The first case we studied, of course, arises when considering the Cauchy problem for the Schr\"odinger equation in the case of a free particle, see \cite{ACSST17A}:
\begin{equation}\label{CPSR}
i\frac{\partial \psi(x,t)}{\partial t}=-\frac{\partial^2 \psi(x,t)}{\partial x^2},\ \ \ \ \psi(x,0)=F_N(x,a).
\end{equation}
As one can immediately verify that the solution $\psi_N(x,t)$, is given by
\begin{equation}\label{CPSRSol}
\psi_N(x,t)=\sum_{\nu=0}^N{N\choose \nu}\left(\frac{1-a}{2}\right)^{\nu}\left(\frac{1+a}{2}\right)^{N-\nu} e^{i(1-2\nu /N) x } e^{-it(1-2\nu /N)^2}.
\end{equation}
This simple example shows that a theory of superoscillatory functions in several variables exists.

Indeed, in the papers \cite{ACJSSST22,CPSS23}, we have extended the aforementioned example by considering analytic functions in one variable, denoted as $G_1,\ldots, G_m$ for $m\geq2$, whose Taylor series at zero have a radius of convergence greater than or equal to 1. Consequently, we define general superoscillating functions of several variables as expressions of the form
$$
F_N(x_1,x_2,\ldots,x_d):=\sum_{\nu=0}^N Z_\nu(N,a)e^{ix_1 G_1(h(N,\nu))}e^{ix_2G_2(h(N,\nu))} \ldots e^{ix_m G_m(h(N,\nu))}
$$
Here, $Z_\nu(N,a)$, for $j=0,...,N$ and $N\in \mathbb{N}_0$, represents suitable coefficients of a superoscillating function in one variable.
We have established conditions on the functions $G_1,\ldots, G_m$ such that
$$
\lim_{N\to \infty}F_N(x_1,x_2,\ldots,x_d)=e^{ix_1 G_1(a) }e^{ix_2 G_2(a) }\ldots e^{ix_m G_m(a)},
$$
ensuring that when $|a|>1$, $F_N(x_1,x_2,\ldots,x_m)$ exhibits superoscillatory behavior. Additionally, we will address the scenario of sequences that admit a supershift in $m\geq 2$ variables.

A key development in the last couple of year has been the understanding of superoscillations and analyticity which was developed in the papers \cite{cssy23,IRREGULAR-SAMP}.
In those papers we make explicit the connection between the phenomenon of supershift and different interpolation techniques.
In particular, we use two interpolation theories namely the Legendre and the Bernstein polynomial interpolation.

More specifically, we use a classical result due to Serge Bernstein to
show that real analyticity for a complex valued function implies a strong form of supershift. On the other hand, a
parametric version of a result by
Leonid Kantorovitch shows that the converse is not true.
We also show that the restriction to $\R$ of any entire function displays supershift, whereas the converse is, in general, not true.

In this paper we push forward this analysis and we make explicitly the connection between supershift and the expected value of a family of independent random variables. This fact allows us to extend the ideas in \cite{cssy23,IRREGULAR-SAMP} to the case of several variables, using once again both Legendre and Bernstein approximation. This is a complete novel approach to the study of supershift in several variables and like in the variable case we are able to recuperate between analyticity and supershift.

The plan of the paper is as follows. In Section 2 we make explicit the relation between the classical superoscillating sequence and random variables via Bernstein approximation. In Section 3 we devote our attention to the specific case of several variables of the approximation results of Serge Bernstein. Multi-variable Lagrange interpolation is discussed and utilized to once again obtain a supershift phenomenon in the multi-variable case. Key tools are the multi-variable residue formula and the new notion of $\mathbf{T}$-predictability which provides a natural extension to this framework of the condition to establish supershift in one variable. In the course of the proof of Theorem \ref{sect3-thm1} we discover an autosimilarity phenomenon that we call Bernstein $\bfrho$-pseudo-autosimilarity. In Section 5 we apply these ideas to some examples of interest in harmonic analysis where we consider complex-valued signals and the $d$-variables vary in the frequency domain or the space domain or the scale domain. The examples are elaborated starting from extensions from the real to the complex case of elementary atoms in Fourier, Gabor, Fourier-Wigner-Ville analysis and also time-scale-frequency analysis. In particular, we show in Proposition \ref {SECT5-prop2} how a translated, scaled and modulated atom, like the Shannon's or Meyer's wavelet, evolves via the Schr\"odinger equation  of the free particle.

\section{Bernstein approximation and probabilities }\label{sect1}

For each $\bfa = (a_1,...,a_m) \in [-1,1]^m$, let $(X_{a_1,n})_{n\geq 1}$,...
, $(X_{a_m,n})_{n\geq 1}$ be $m$ independent sequences of independent random variables with
\begin{equation*}
\mathbb P(\{X_{a_j,n}=-1\}) = \frac{1-a_j}{2},\quad \mathbb P(\{X_{a_j,n}=1\}) = \frac{1+a_j}{2}
\end{equation*}
for any $j=1,...,m$ and $n\in \N^*$.
Let also $({\bf R},{\bf d})$ be a metric space.
For any continuous function $\phi: [-1,1]^m \times {\bf R} \rightarrow \C$, any $\bfN = (N_1,...,N_m) \in (\N^*)^m$ and any $x\in {\bf R}$, one has then
\begin{equation}\label{sect1-eq1}
\mathbb E \Big(\phi\, \Big( \frac{\bfS_{\bfa,\bfN}}{\bfN}, x \Big)\Big) =
\sum_{\bfnu \prec \bfN} \binom{\bfN}{\bfnu}
\, \Big(\frac{1-\bfa}{2}\Big)^{\bfnu} \, \Big(\frac{1+\bfa}{2}\Big)^{\bfN - \bfnu}\, \phi
\Big( {\bf 1} - 2\, \frac{\bfnu}{\bfN}\Big),
\end{equation}
where $\bfnu \prec \bfN$ means $0\leq \nu_j \leq N_j$ for $j=1,...,m$,
$$
{\bf 1} - 2 \, \bfnu/\bfN := (1-2 \, \nu_1/N_1,...,1-2\, \nu_m/N_m),
$$
and
\begin{equation}\label{sect1-eq2}
\begin{split}
S_{a_j,N_j}:= \sum_{n=1}^{N_j} X_{a_j,n} \quad (j=1,...,m)\ &,\
\frac{\bfS_{\bfa,\bfN}}{\bfN} := \Big( \frac{S_{a_1,N_1}}{N_1},...,\frac{S_{a_m,N_m}}{N_m}\Big),  \\
\binom{\bfN}{\bfnu}
\, \Big(\frac{1-\bfa}{2}\Big)^{\bfnu} \, \Big(\frac{1+\bfa}{2}\Big)^{\bfN - \bfnu} & :=
\prod_{j=1}^m  \binom{N_j}{\nu_j}
\, \Big(\frac{1-a_j}{2}\Big)^{\nu_j} \, \Big(\frac{1+a_j}{2}\Big)^{N_j - \nu_j}.
\end{split}
\end{equation}
Since $\mathbb E[X_{a_j,\nu_j}] = a_j$ and $\sigma^2 (X_{a_j,\nu_j}) = 1 -a_j^2$ for any $j=1,...,m$ and $\nu_j = 1,...,N_j$, Chebychev's inequality implies that for any $\eta>0$, $j=1,...,m$, and $N_j\in \N^*$
$$
\mathbb P \Big(
\Big\{\Big|
\frac{S_{a_j,N_j}}{N_j} - a_j\Big|
\geq \eta \Big\}\Big) \leq \frac{ N_j\, ( 1-a_j^2)}{ N^2_j\, \eta^2} = \frac{1}{N_j}\, \frac{1}{\eta^2}.
$$
Given a compact subset ${\rm K} \subset {\bf R}$, one has then for any $\bfa \in [-1,1]^m$ and
$x\in {\rm K}$ that
\begin{equation*}
\begin{split}
\Big|\mathbb E \Big(\phi\, \Big(\frac{\bfS_{\bfa,\bfN}}{\bfN}, x \Big)\Big)- \phi\, (\bfa,x)\Big| & \leq \mathbb E \Big( \Big|  \phi\, \Big(\frac{\bfS_{\bfa,\bfN}}{\bfN} , x \Big) - \phi\, (\bfa,x)\Big|
\, {\bf 1}_{\bigcap_{1\leq j\leq m} \big\{\big|\frac{S_{a_j,N_j}}{N_j} - a_j\big| \leq \eta\big\}}\Big) \\
&\qquad  + 2 \sup_{\stackrel{\bfa \in [-1,1]^m}{x\in {\rm K}}}
|\phi\, (\bfa,x)|\, \sum_{j=1}^m \mathbb P \Big(
\Big\{\, \Big|
\frac{S_{a_j,N_j}}{N_j} - a_j\Big|
\geq \eta \Big\}\Big) \\
& \leq \mathbb E \Big( \Big|  \phi\, \Big(\frac{\bfS_{\bfa,\bfN}}{\bfN} , x \Big) - \phi\, (\bfa,x)\Big|
\, {\bf 1}_{\bigcap_{1\leq j\leq m} \big\{\big|\frac{S_{a_j,N_j}}{N_j} - a_j\big| \leq \eta\big\}}\Big) \\
& \qquad + 2\, \frac{m}{\eta^2\, \min_j N_j}\, \sup_{\stackrel{\bfa \in [-1,1]^m}{x\in {\rm K}}}
|\phi\, (\bfa,x)|.
\end{split}
\end{equation*}
It follows then from the uniform continuity of $\phi$ on $[-1,1]^m \times {\rm K}$, where
${\rm K} \subset \subset {\bf R}$, that
\begin{equation}\label{sect1-eq2bis}
\lim\limits_{\min_{j} N_j \rightarrow +\infty}
\sup\limits_{\stackrel{\bfa \in [-1,1]^m}{x
\in {\rm K}}} \Big|\mathbb E \Big(\phi\, \Big(\frac{\bfS_{\bfa,\bfN}}{\bfN}, x \Big)\Big)- \phi\, (\bfa,x)\Big| =0,
\end{equation}
which is a fundamental result established by S. Bernstein in \cite{Bern12}, see also
\cite[\textsection 3]{Kow16} for the multivariate approach described here.
It follows in particular from \eqref{sect1-eq1} that for any $\bfa \in [-1,1]^m$ and $\bfkappa \in \N^m$
\begin{equation}\label{sect1-eq3}
\begin{split}
\mathbb E \Big(\Big( \frac{\bfS_{\bfa,\bfN}}{\bfN}\Big)^{\bfkappa}\Big)
& = \mathbb E \Big(\prod_{j=1}^m
\Big(\frac{S_{a_j,N_j}}{N_j}\Big)^{\kappa_j}\Big)  =
\prod_{j=1}^m \mathbb E
\Big(\Big(\frac{S_{a_j,N_j}}{N_j}\Big)^{\kappa_j}\Big) \\
& =  \prod_{j=1}^m
\Big(\sum_{\nu = 0}^{N_j} \binom{N_j}{\nu}
\, \Big(\frac{1-a_j}{2}\Big)^{\nu} \, \Big(\frac{1+a_j}{2}\Big)^{N_j - \nu}
\Big( 1 - 2\, \frac{\nu}{N_j}\Big)^{\kappa_j}\Big),
\end{split}
\end{equation}
so that the following identity in $\R[[\bfz]] = \R[[z_1,...,z_m]]$ holds for any $\bfw \in [-1,1]^m$:
\begin{equation}\label{sect1-eq4}
\begin{split}
\sum_{\bfkappa \in \N^m}
\mathbb E \Big(\Big( \frac{\bfS_{\bfw,\bfN}}{\bfN}\Big)^{\bfkappa}\Big)\, \frac{\bfz^{\bfkappa}}{\bfkappa!}
& = \prod_{j=1}^m \Big(\cosh \frac{z_j}{N_j} + w_j\, \sinh \, \frac{z_j}{N_j}\Big)^{N_j} \\
& = \prod_{j=1}^m \Big( \sum_{\kappa =0}^\infty \Big(\frac{1+ w_j}{2} + (-1)^\kappa \frac{1-w_j}{2}\Big)\,
\frac{1}{\kappa!}\, \Big(\frac{z_j}{N_j}\Big)^\kappa\Big)^{N_j}.
\end{split}
\end{equation}
Observe that the right-hand side in the first line of \eqref{sect1-eq4} defines, as a function of $\bfw$ and $\bfz$, both now considered in $\C^m$, an entire function in $2m$ variables.
Since
$$
\Big|\frac{1+ w_j}{2} + (-1)^\kappa \frac{1-w_j}{2}\Big|
\leq (\max (1,|w_j|))^\kappa
$$
for any $j=1,...,m$ and any $\kappa \in \N$, it follows from the second equality in \eqref{sect1-eq4} that
the analytic continuation to $\C^m$ of
$$
\bfa \in [-1,1]^m  \longmapsto  \mathbb E \Big(\Big( \frac{\bfS_{\bfa,\bfN}}{\bfN}\Big)^{\bfkappa}\Big),
$$
namely the polynomial map
$$
\bfw \in \C^m \longmapsto \mathbb B_{(\cdot)^{\bfkappa}}^{\bfN} \, (\bfw) = \prod_{j=1}^m
\Big(\sum_{\nu = 0}^{N_j} \binom{N_j}{\nu}
\, \Big(\frac{1-w_j}{2}\Big)^{\nu} \, \Big(\frac{1+w_j}{2}\Big)^{N_j - \nu}
\Big( 1 - 2\, \frac{\nu}{N_j}\Big)^{\kappa_j}\Big)
$$
satisfies for any $\bfa \in \C^m$ and any $\bfN\in (\N^*)^m$
\begin{equation}\label{sect1-eq5}
\Big|\frac{ \mathbb B_{(\cdot)^{\bfkappa}}^{\bfN} \, (\bfw)}{\bfkappa !}\Big|
\leq \frac{1}{\bfkappa!} \prod_{j=1}^d (\max (1,|w_j|))^{\kappa_j} \Longrightarrow
\big|\mathbb B_{(\cdot)^{\bfkappa}}^{\bfN} \, (\bfw)\big| \leq \prod_{j=1}^m (\max (1,|w_j|))^{\kappa_j}.
\end{equation}

\section{Bernstein and Lagrange approximation in $\C^m$ }\label{sect2}

Let $({\bf R}, {\bf d})$ be a metric space. The following result extends to the multivariate parametric context a result due to Serge Bernstein in the univariate and parametric-free context \cite{Bern35}.

\begin{theorem}\label{sect2-thm1} Let $\bfr = (r_1,...,r_m)\in (]1,+\infty])^m$ and $\phi~: (\bfa,x) \longmapsto
\phi\, (\bfa,x) \in \C$ be a continuous function on $\big(\prod_{j=1}^m ]-r_j,r_j[\big) \times {\bf R}$. Suppose that there exists a continuous map $x \longmapsto f_x$ from ${\bf R}$
to $H(\mathbb D^m_{\bfr})$, where $\mathbb D^m_{\bfr} := \{\bfw \in \C^m\,:\, |w_j|< r_j\ \text{for}\ j=1,...,m\}$, such for each $x\in {\bf R}$, $\bfa \in \prod_{j=1}^m ]-r_j,r_j[ \longmapsto \phi\, (\bfa,x)$ is the restriction of $\bfw \longmapsto f_x (\bfw)$ to $\prod_{j=1}^m ]-r_j,r_j[$. Then
\begin{equation}\label{sect2-eq0}
\phi\, (\bfa,x) = \lim\limits_{\{\bfN\in (\N^*)^m\,:\, \min \bfN \rightarrow +\infty\}}
\sum_{\bfnu \prec \bfN} \binom{\bfN}{\bfnu}
\, \Big(\frac{1-\bfa}{2}\Big)^{\bfnu} \, \Big(\frac{1+\bfa}{2}\Big)^{\bfN - \bfnu}\, \phi\, \Big({\bf 1} - 2\, \frac{\bfnu}{\bfN},x\Big)
\end{equation}
uniformly in $(\bfa,x)$ on any compact subset of $\big(\prod_{j=1}^m ]-r_j,r_j[\big) \times
{\bf R}$.
\end{theorem}

\begin{proof}
Let
$$
f_x(\bfw)= \sum\limits_{\bfkappa \in \N^m} \gamma_{\bfkappa} (x)\, \bfw^\kappa
$$
be the Taylor expansion of $\bfw \mapsto f(\bfw,x)$ about the origin in $\C^m$. For any
$\bfrho \in \prod_{j=1}^m [0,r_j[$ and ${\rm K} \subset \subset {\bf R}$, it follows from the hypothesis, together with Cauchy inequalities, that
\begin{equation}\label{sect2-eq1}
M (\bfrho,{\rm K}) = \sup_{x \in {\rm K}} \sum_{\kappa \in \N^m} |\gamma_{\bfkappa} (x)|\,
\prod_{j=1}^m (\max (1,\rho_j))^{\kappa_j} < +\infty.
\end{equation}
If one plots in the right-hand side of \eqref{sect2-eq1} the inequalities \eqref{sect1-eq5}, one gets that
\begin{equation}\label{sect2-eq1bis}
\sup\limits_{\stackrel{\bfN \in (\N^*)^m}{x \in {\rm K}}}
\sum\limits_{\bfkappa \in \N^m} |\gamma_{\bfkappa} (x)| \,
\big|\mathbb B_{(\cdot)^{\bfkappa}}^{\bfN} \, (\bfw)\big| \leq M(\bfrho,{\rm K}).
\end{equation}
As a consequence, all multi-series (of functions in the variables $\bfw$ and $x$)
\begin{equation}\label{sect2-eq2}
\begin{split}
\sum_{\bfkappa \in \N^m} \gamma_{\bfkappa} (x)\, \mathbb B_{(\cdot)^{\bfkappa}}^{\bfN} \, (\bfw)
& = \sum_{\bfnu \prec \bfN} \binom{\bfN}{\bfnu}
\, \Big(\frac{1-\bfw}{2}\Big)^{\bfnu} \, \Big(\frac{1+\bfw}{2}\Big)^{\bfN - \bfnu}
\, \Big( \sum_{\bfkappa \in \N^m}
\gamma_{\bfkappa} (x) \,
\Big({\bf 1} - 2\, \frac{\bfnu}{\bfN}\Big)^{\bfkappa}\Big) \\
& = \sum_{\bfnu \prec \bfN} \binom{\bfN}{\bfnu}
\, \Big(\frac{1-\bfw}{2}\Big)^{\bfnu} \, \Big(\frac{1+\bfw}{2}\Big)^{\bfN - \bfnu}\, \phi\, \Big({\bf 1} - 2\, \frac{\bfnu}{\bfN},x\Big),
\end{split}
\end{equation}
where $\bfN \in (\N^*)^m$, converge normally on every compact subset of $\mathbb D^m_{\bfr} \times {\bf R}$.
Moreover, one has
\begin{equation}\label{sect2-eq3}
\Big|\sum_{\bfnu \prec \bfN} \binom{\bfN}{\bfnu}
\, \Big(\frac{1-\bfw}{2}\Big)^{\bfnu} \, \Big(\frac{1+\bfw}{2}\Big)^{\bfN - \bfnu}\, \phi\, \Big({\bf 1} - 2\, \frac{\bfnu}{\bfN},x\Big)\Big|
\leq M(|\bfw|, {\rm K})
\end{equation}
for any $\bfN\in (\N^*)^m$ and any ${\rm K} \subset \subset {\bf R}$, where $|\bfw| = (|w_1|,...,|w_m|)$.
This implies that the family
\begin{equation}\label{sect2-eq3bis}
\Big\{w \longmapsto \sum_{\bfnu \prec \bfN} \binom{\bfN}{\bfnu}
\, \Big(\frac{1-\bfw}{2}\Big)^{\bfnu} \, \Big(\frac{1+\bfw}{2}\Big)^{\bfN - \bfnu}\, \phi\, \Big({\bf 1} - 2\, \frac{\bfnu}{\bfN},x\Big)\,:\,
\bfN \in (\N^*)^m,\ x \in {\rm K}\Big\}
\end{equation}
is bounded in $H(\mathbb D_{\bfr}^m)$.
Let $\max ({\bf 1},\bfrho) \prec \bfrho\,'\prec \bfr$, which means $\max (1,\rho_j) < \rho_j\,' < r_j$ for $j=1,...,m$,
${\rm K} \subset \subset {\bf R}$ and $\varepsilon>0$.
Since the map $x \in {\bf R} \longmapsto f_x
\in H(\mathbb D^m_\bfr)$ is uniformly continuous on any compact subset of ${\bf R}$, there exists
$\eta\, ({\rm K},\varepsilon)>0$ such that
\begin{equation}\label{sect2-eq3ter}
x, {\rm y} \in {\rm K}\ \text{and}\ {\bf d}\, (x, {\rm y}) < \eta \, ({\rm K},\varepsilon)
\Longrightarrow \sum\limits_{\kappa \in \N^m}
|\gamma_{\kappa} (x) - \gamma_\kappa ({\rm y})|^2\,
(\bfrho')^{2\bfkappa} \leq \varepsilon^2.
\end{equation}
For any $\bfw \in \overline{\mathbb D^m_{\bfrho}}$, any $x,{\rm y} \in {\rm K}$ with
${\bf d}\, (x,{\rm y}) < \eta\, ({\rm K}, \varepsilon)$ and any $\bfN\in (\N^*)^m$, it follows from
\eqref{sect2-eq3ter} and the previous considerations that
\begin{multline*}
\Big|
\sum_{\bfnu \prec \bfN} \binom{\bfN}{\bfnu}
\, \Big(\frac{1-\bfw}{2}\Big)^{\bfnu} \, \Big(\frac{1+\bfw}{2}\Big)^{\bfN - \bfnu}\,
\Big( \phi\, \Big({\bf 1} - 2\, \frac{\bfnu}{\bfN},x\Big) -
\phi\, \Big({\bf 1} - 2\, \frac{\bfnu}{\bfN},{\rm y}\Big)\Big)\Big| \\
\leq \varepsilon \, \Big({\sum\limits_{\bfkappa \in \N^m} \Big(\frac{\max ({\bf 1},\bfrho)}{\bfrho'}\Big)^{2\bfkappa}}\Big)^{1/2} = \varepsilon  \, \Big({\prod_{j=1}^m \frac{1}{1 - ((\max(1,\rho_j))/\rho_j\,')^2}}\Big)^{1/2}.
\end{multline*}
Since $[-1,1]^m$ is not contained in any complex hypersurface of $\C^m$ and
\begin{equation}\label{sect2-eq4}
\lim\limits_{\min_{j} N_j \rightarrow +\infty}
\sum_{\bfnu \prec \bfN} \binom{\bfN}{\bfnu}
\, \Big(\frac{1-\bfw}{2}\Big)^{\bfnu} \, \Big(\frac{1+\bfw}{2}\Big)^{\bfN - \bfnu}\, \phi\, \Big({\bf 1} - 2\, \frac{\bfnu}{\bfN},x\Big) = \bfphi(\bfw,x)
\end{equation}
uniformly in $(\bfw,x)$ on any compact subset of $[-1,1]^m \times {\bf R}$ according to
\eqref{sect1-eq2bis}, it follows from Vitali-Montel theorem in the multivariate setting, see for example
\cite[\textsection 1]{DMBH16}, that \eqref{sect2-eq4} holds uniformly in $(\bfw,x)$ on any compact subset of
$\mathbb D_{\bfr}^m \times {\bf R}$. It remains to restrict such convergence to the intersection of
$\mathbb D_{\bfr}^m \times {\bf R}$ with $\R^d \times {\bf R}$, namely to
$\big(\prod_{j=1}^m ]-r_j,r_j[\big) \times {\bf R}$, in order to get the required assertion.
\end{proof}

\begin{example}\label{sect2-expl1}
{\rm Let $m=1$, $G_1,...,G_d$ be univariate holomorphic functions respectively in the discs $\mathbb D\,(0,R_1)$,...,$\mathbb D\,(0,R_d)$, $U_{d,\boldsymbol R} := \{(a,x)\in \R^{d+1}\,:\, |a\, x_\ell| < R_\ell \ \text{for}\ \ell =1,...,d\}$.
Theorem \ref{sect2-thm1} (when $m=1$) implies that the sequence of functions
$$
\Big( (a,x) \in U_{d,\boldsymbol R} \longmapsto
\sum\limits_{\nu=0}^N
\binom{N}{\nu} \, \Big(\frac{1- a}{2}\Big)^\nu \Big(\frac{1+a}{2}\Big)^{N-\nu}\, \prod_{\ell=1}^d
G_\ell\, ((1-2\, \nu/N)\, x_\ell) \Big)_{N\geq 1}
$$
converges uniformly on any compact subset of $U_{d,\boldsymbol R}$ to 
$(a,x) \longmapsto \prod_{\ell=1}^d G_\ell (a\, x_\ell)$. This corresponds to
\cite[Example 4.9 (1)]{CPSS23}. More generally, let $q_1,...,q_d \in \N^*$ and
$$
U_{d,\boldsymbol q,\boldsymbol R} := \big\{(a,x)\in \R^{d+1}\,:\, |a|^{q_\ell}\, |x_\ell|< R_\ell\ \text{for}\ \ell =1,...,d\big\}.
$$
Theorem \ref{sect2-thm1} (still when $m=1$) implies that the sequence of continuous functions
$$
\Big( (a,x) \in U_{d,\boldsymbol q,\boldsymbol R} \longmapsto
\sum\limits_{\nu=0}^N
\binom{N}{\nu} \, \Big(\frac{1- a}{2}\Big)^\nu \Big(\frac{1+ a}{2}\Big)^{N-\nu}\, \prod_{\ell=1}^d
G_\ell\, ((1-2\, \nu/N)^{q_\ell}\, x_\ell) \Big)_{N\geq 1}
$$
converges uniformly on any compact subset of $U_{d,\boldsymbol q,\boldsymbol R}$ to the continuous function
$(a,x) \longmapsto \prod_{\ell=1}^d G_\ell (a \, x_\ell)$, see \cite[Theorem 3.9]{ACJSSST22} for what concerns the particular important case where $R_\ell =+\infty$, $G_\ell (w) = \exp (iw)$ for any $\ell$.}
\end{example}

For each $\bfN =(N_1,...,N_m) \in (\N^*)^m$ and $j=1,...,m$, let $(h_j(N_j,\nu))_{0\leq \nu \leq N_j}$ be a strictly increasing sequence of elements in $[-1,1]$. For each $\bfnu\prec \bfN$, let
$$
\bfh (\bfN,\bfnu) := (h_1(N_1,\nu_1),...,h_m(N_m,\nu_m)) \in [-1,1]^m.
$$
The following result also holds.

\begin{theorem}\label{sect2-thm2} Let $\bfr = (r_1,...,r_m)\in (]1,+\infty])^m$ and $\phi~: (\bfa,x) \longmapsto
\phi\, (\bfa,x) \in \C$ be a continuous function on $\big(\prod_{j=1}^m ]-r_j,r_j[\big) \times {\bf R}$. Suppose that, for each
$x\in {\bf R}$, $\bfa \longmapsto \phi(\bfa,x)$ is the restriction to $\prod_{j=1}^m ]-r_j,r_j[$ of a complex valued function $f_x$ which is defined and holomorphic in the polydisc
$$
\mathbb D_{\bfr + {\bf 2}}^m =
\{\bfw \in \C^d\,:\, |w_j|< r_j+2\ \text{for}\ j=1,...,m\},
$$
such that, for any ${\rm K}\subset \subset {\bf R}$, $\{f_x\,:\, x \in {\rm K}\}$ is a bounded family in $H (\mathbb D^m_{\bfr + {\bf 2}})$. Then
\begin{equation}\label{sect2-eq5}
\phi\, (\bfa,x) = \lim\limits_{\{\bfN \in (\N^*)^m\,:\, \min \bfN\rightarrow +\infty\}}
\sum\limits_{\bfnu \prec \bfN} \prod_{j=1}^m \prod_{\stackrel{\nu=0}{\nu\not=\nu_j}}^{N_j}
\frac{a_j - h_j(N_j,\nu)}{h_j(N_j,\nu_j) - h_j(N_j,\nu)}\, \phi\,  ({\bfh} (\bfN,\bfnu),x)
\end{equation}
uniformly in $(\bfa,x)$ on any compact subset of $\big(\prod_{j=1}^m ]-r_j,r_j[\big) \times
{\bf R}$.
\end{theorem}

\begin{proof} We prove the result inductively with respect to $m\in \N^*$. Namely, we prove first assertion $(A_1)$, then assertion
$(A_m)$ assuming that all assertions $(A_\mu)$ for $1\leq \mu <m$ are true. \\
Consider first the case $m=1$,
$\bfr = r_1 = r$. Let $\eta \in ]0,r[$. The hypothesis imply that, for any ${\rm K} \subset \subset {\bf R}$,
$M_{\eta} ({\rm K}) := \sup \,
\big\{|f_x (w)|\,:\, w\in \C\ {\rm with}\ |w| = r + 2 - \eta/2,\ x \in {\rm K}\big\} < +\infty$.
It follows from Cauchy formula that for any $a \in [-(r-\eta),r-\eta]$ and $x \in {\rm K}$
\begin{equation}\label{sect2-aux-eq1}
\phi\, (a, x) = f_x (a) =  \frac{1}{(2i\pi)}
\int_{\Gamma_{r + 2-\eta/2}}
f_x (w)\, \frac{\big((Q_{N}(w) - Q_{N} (a)) + Q_{N} (a)\big)}{(w-a)} \, \frac{dw}{Q_{N}(w)},
\end{equation}
where $Q_{N} (w) := \prod_{\nu=0}^{N} (w - h(N,\nu))$ for $N\in \N^*$
and $\Gamma_{r+2 -\eta/2}$ is the Shilov boundary of the disc $\D_{r+2-\eta/2} := \D(0,r+2-\eta/2)$. Residue formula leads to
\begin{multline}\label{sect2-aux-eq2}
\frac{1}{(2i\pi)}
\int_{\Gamma_{r + 2-\eta/2}}
f_x (w)\, \frac{Q_{N}(w) - Q_{N} (a)}{w-a} \, \frac{dw}{Q_{N}(w)} \\ =
\sum\limits_{\nu=0}^N \prod_{\stackrel{\nu'=0}{\nu'\not= \nu}}^N \frac{a - h(N,\nu')}{h(N,\nu) - h(N,\nu')}\, \phi\,(h(N,\nu),x),
\end{multline}
which is the expression on the right-hand side of \eqref{sect2-eq5} when $m=1$.
On the other hand,
\begin{multline}\label{sect2-aux-eq3}
\Big|
 \frac{1}{(2i\pi)}
\int_{\Gamma_{r+2-\eta/2}} f_x (w)\, \frac{Q_N(a)}{Q_N(w)}\, \frac{dw}{w-a}\Big| \\
\leq M_{\eta} ({\rm K}) \, \frac{r+2 - \eta/2}{2+ \eta/2}\, \Big(\frac{r+1-\eta}{r+1-\eta/2}\Big)^{N+1} = O_{\eta,{\rm K}}\, (e^{-\varepsilon_\eta N})
\end{multline}
for some $\varepsilon_\eta >0$. Since $\eta$ can be chosen arbitrary small, \eqref{sect2-eq5} when $d=1$ follows from the integral representation formula \eqref{sect2-aux-eq1}, once combined with residue formula \eqref{sect2-aux-eq2} and upper error estimates \eqref{sect2-aux-eq3}. \\
Assume now that Theorem \ref{sect2-thm2} is proved up to the step $m-1$ ($m\geq 2$).
Let $\bfeta
\in \prod_{j=1}^m ]0,r_j[$. One has from the hypothesis $(H_m)$ that for any ${\rm K} \subset \subset {\bf R}$
$$
M_{\bfeta} ({\rm K}) := \sup \,
\big\{|f_x (\bfw)|\,:\, \bfw\in \C^m\ \text{with}\  |w_j| = r_j + 2 - \eta_j/2\ \text{for}\ j=1,...,m,\ x \in {\rm K}\big\}
< +\infty.
$$
Let $J\subset \{1,...,m\}$ be a proper non-empty subset of $\{1,...,m\}$ and denote, for
$\bfN \in (\N^*)^m$, $\bfa \in \prod_{j=1}^m ]-r_j,r_j[$, $\bfw \in \mathbb D^m_{\bfr + {\bf2}}$,
$\bfN_J := (N_j)_{j\in J}$, $\bfa_J:= (a_j)_{j\in J}$, $\bfw_{J'}= (w_j)_{j\not\in J}$ (with increasing order on indexes). Let also $\Gamma_{\bfr_{J'} + {\bf 2} - \bfeta_{J'}/2}$ be the Shilov boundary of the polydisc
$$
\mathbb D^{m - \#J}_{\bfr_{J'} + 2 - \bfeta_{J'}/2} : = \{\bfw_{J'} \,:\, |w_j| = r_j + 2 - \eta/2\ \text{for}\ j\not\in J\} \subset \C^{m-\#J}
$$
and $|\Gamma_{\bfr_{J'} + {\bf 2} - \bfeta_{J'}/2}|$ be its support.
If
\begin{equation}\label{sect2-aux-eq4}
\Phi_{J,\bfN}\, (\bfa_J,\bfw_{J'},x) :=
\sum\limits_{\bfnu \prec \bfN_J} \prod_{j\in J} \prod_{\stackrel{\nu =0}{\nu \not=\nu_j}}^{N_j}
\frac{a_j - h_j(N_j,\nu)}{h_j(N_j,\nu_j) - h_j(N_j,\nu)}\, f_x\,  ({\bfh} (\bfN_J,\bfnu), \bfw_{J'})
\end{equation}
when $|a_j|< r_j$ for $j\in J$, $|w_j|< r_j + 2$ for $j\not\in J$ and $x\in {\bf R}$, it follows from the inductive assertion $(A_{m-\#J})$
that
\begin{multline}\label{sect2-aux-eq5}
\sup \big\{ |\Phi_{J,\bfN}\, (\bfa_J,\bfw_{J'},x)|\,:\,
\bfa_J\in \prod_{j\in J}
[-(r_j-\eta_j),r_j-\eta_j],\  \bfw_{J'}\in |\Gamma_{\bfr_{J'} + {\bf 2} - \bfeta_{J'}/2}|,\
x \in {\rm K}\big\} \\
:= M_{J,\bfeta} ({\rm K}) < +\infty.
\end{multline}
We complete now the proof of Theorem \ref{sect2-thm2} when $m\geq 2$ (assuming $(A_\mu)$ for $\mu<m$).
For any $\bfa \in \prod_{j=1}^m [-(r_j-\eta_j),r_j-\eta_j]$ and $x \in {\rm K}$, Cauchy formula implies in the multivariate (tensorial) setting that
\begin{equation}\label{sect2-aux-eq6}
\phi\, (\bfa, x) = \frac{1}{(2i\pi)^m}
\int_{\Gamma_{\bfr + {\bf 2}-\bfeta/2}}
f_x (\bfw) \frac{\prod_{j=1}^m \big(Q_{N_j}(w_j) - Q_{N_j} (a_j) + Q_{N_j} (a_j)\big)}{(w_1-a_1)\cdots (w_m-a_m)} \frac{dw_1 \wedge \cdots \wedge dw_m}{Q_{\bfN}(\bfw)},
\end{equation}
where
$$
Q_{N_j} (w_j) := \prod\limits_{\nu=0}^{N_j} (w_j - h_j(N_j,\nu))\ ,\
Q_{\bfN} (\bfw) := \prod\limits_{j=1}^m Q_{N_j} (w_j),
$$
and $\Gamma_{\bfr  + {\bf 2}- \bfeta/2}$ denotes the Shilov boundary of the polydisc $\mathbb D^m_{\bfr + {\bf 2}- \bfeta/2}$. Residue formula in such setting leads to
\begin{multline}\label{sect2-aux-eq7}
\frac{1}{(2i\pi)^m}
\int_{\Gamma_{\bfr + {\bf 2} - \bfeta/2}}
f_x (\bfw) \, \Big(\prod_{j=1}^m
\frac{Q_{N_j} (w_j) - Q_{N_j} (a_j)}{w_j-a_j}\Big)\, \frac{dw_1 \wedge \cdots \wedge dw_m}{Q_{\bfN}(\bfw)} \\
= \sum\limits_{\bfnu \prec \bfN} \prod_{j=1}^m \prod_{\stackrel{\nu=0}{\nu\not=\nu_j}}^{N_j}
\frac{a_j - h_j(N_j,\nu)}{h_j(N_j,\nu_j) - h_j(N_j,\nu)}\, \phi\,  ({\bfh} (\bfN,\bfnu),x).
\end{multline}
Given any proper non-empty subset $J$ of $\{1,...,m\}$, observe that it also implies that
\begin{multline}\label{sect2-aux-eq8}
\int_{\Gamma_{\bfr + {\bf 2} - \bfeta/2}}
f_x (\bfw) \Big( \prod_{j\in J}
\frac{Q_{N_j} (w_j) - Q_{N_j} (a_j)}{w_j -a_j}\, \frac{1}{Q_{N_j} (w_j)}\Big)\,
\prod_{j\not\in J}
\frac{Q_{N_j} (a_j)}{Q_{N_j} (w_j)}\, \frac{\bigwedge_{j=1}^m dw_j}{\prod_{j\not\in J}
(w_j-a_j)}
\\
= (2i\pi)^{\#J} \int_{\Gamma_{\bfr_{J'} + {\bf 2} - \bfeta_{J'}/2}} \Phi_{J,\bfN}\, (\bfa_J,\bfw_{J'}, x)
\, \prod_{j\not \in J}
\frac{Q_{N_j} (a_j)}{Q_{N_j} (w_j)}\, \bigwedge_{j\not\in J} \frac{dw_j}{w_j-a_j}.
\end{multline}
On the other hand, for any $\bfw$ on the support of $\Gamma_{\bfr + {\bf 2}-\bfeta/2}$ and any $j=1,...,m$
\begin{equation}\label{sect2-aux-eq9}
\Big|
\frac{Q_{N_j} (a_j)}{Q_{N_j} (w_j)}\Big|
= \prod_{\nu =0}^{N_j} \Big|\frac{a_j - h_j (N_j,\nu)}{w_j - h_j(N_j,\nu)}\Big| \leq
\Big(\frac{1 + r_j -\eta_j}{1+r_j-\eta_j/2}\Big)^{N_j+1}.
\end{equation}
It follows from \eqref{sect2-aux-eq8}, combined with \eqref{sect2-aux-eq5} and
\eqref{sect2-aux-eq9}, that for any non-empty proper subset $J$ of $\{1,...,m\}$
\begin{multline}\label{sect2-aux-eq10}
\Big|\int_{\Gamma_{\bfr + {\bf 2} - \bfeta/2}}
f_x (\bfw) \Big( \prod_{j\in J}
\frac{Q_{N_j} (w_j) - Q_{N_j} (a_j)}{w_j -a_j}\, \frac{1}{Q_{N_j} (w_j)}\Big)\,
\prod_{j\not\in J}
\frac{Q_{N_j} (a_j)}{Q_{N_j} (w_j)}\, \frac{\bigwedge_{j=1}^d dw_j}{\prod_{j\not\in J}
(w_j-a_j)}\Big| \\
\leq M_{J,\bfeta} ({\rm K}) \,
\prod_{j\not\in J} \Big(\frac{r_j + 2 -\eta_j}{2 + \eta_j/2}\Big)\,
\Big(\frac{1 + r_j -\eta_j}{1+r_j-\eta_j/2}\Big)^{N_j+1} = O_{J,\bfeta,{\rm K}} \big(e^{-\varepsilon_{J',\bfeta}
\min_{j\not\in J} N_j}\big)
\end{multline}
for some $\varepsilon_{J',\bfeta}>0$. Since one has also (when $J=\emptyset$)
\begin{multline}\label{sect2-aux-eq11}
\Big| \int_{\Gamma_{\bfr + {\bf 2} - \bfeta/2}}
f_x (\bfw)\,
\frac{Q_{\bfN} (\bfa)}{Q_{\bfN} (\bfw)}\,
\bigwedge_{j=1}^m \frac{dw_j}{w_j-a_j} \Big| \\
\leq M_{\bfeta} ({\rm K}) \, \prod_{j=1}^m \Big(\frac{r_j + 2 -\eta_j}{2 + \eta_j/2}\Big)\,
\Big(\frac{1 + r_j -\eta_j}{1+r_j-\eta_j/2}\Big)^{N_j+1}   =
 O_{\bfeta,{\rm K}} \big(e^{-\varepsilon_{\bfeta}\,
\sum_j N_j}\big)
\end{multline}
for some $\varepsilon_{\bfeta}> 0$, \eqref{sect2-eq5} follows from the integral representation formula \eqref{sect2-aux-eq6}, together with residue formula \eqref{sect2-aux-eq7} and upper estimates \eqref{sect2-aux-eq10} and
\eqref{sect2-aux-eq11}.
\end{proof}

\begin{example}\label{sect2-expl2}
{\rm Let $m=1$, $G_1,...,G_d$ be univariate holomorphic functions respectively in the discs $\mathbb D\,(0,R_1)$,...,$\mathbb D\,(0,R_d)$, $V_{d,\boldsymbol R} := \{(a,x)\in \R^{d+1}\,:\, (|a|+2)\, |x_\ell| < R_\ell \ \text{for}\ \ell =1,...,d\}$.
Theorem \ref{sect2-thm2} (when $m=1$) implies that the sequence of functions
\begin{equation*}
\Big( (a,x) \in V_{d,\boldsymbol R} \longmapsto
\sum\limits_{\nu=0}^N
\Big(\prod_{\stackrel{\nu'=0}{\nu\,'\not=\nu}}^N \frac{a - h(N,\nu\,')}{h(N,\nu) - h(N,\nu\,')}\Big)\,
\prod_{\ell =1}^d G_\ell\, (h(N,\nu)\, x_\ell) \Big)_{N\geq 1}
\end{equation*}
converges uniformly on any compact subset of $V_{d,\boldsymbol R}$ to
$(a,x) \longmapsto \prod_{\ell=1}^d G_\ell (a\, x_\ell)$. This corresponds to \cite[Example 4.9 (2)]{CPSS23}. More generally, let $q_1,...,q_d \in \N^*$ and
$$
V_{d,\boldsymbol q,\boldsymbol R} := \big\{(a,x)\in \R^{d+1}\,:\, (|a| +2)^{q_\ell}\, |x_\ell|< R_\ell\ \text{for}\ \ell =1,...,d\big\}.
$$
Theorem \ref{sect2-thm2} (still when $m=1$) implies that the sequence of continuous functions
$$
\Big( (a,x) \in V_{d,\boldsymbol q,\boldsymbol R} \longmapsto
\sum\limits_{\nu=0}^N
\Big(\prod_{\stackrel{\nu'=0}{\nu\,'\not=\nu}}^N \frac{a - h(N,\nu\,')}{h(N,\nu) - h(N,\nu\,')}\Big)\,
\prod_{\ell =1}^d G_\ell\, (h^{q_\ell}(N,\nu)\, x_\ell) \Big)_{N\geq 1}
$$
converges uniformly on any compact subset of $V_{d,\boldsymbol q,\boldsymbol R}$ to
$(a,x) \longmapsto \prod_{\ell=1}^d G_\ell (a^{q_\ell}\, x_\ell)$.
}
\end{example}

\section{Iterated Bernstein approximation and autosimilarity}\label{sect3}

For any $T \in [0,1[$, let
\begin{equation}\label{sect3-eq1}
\Gamma_T := \big\{w\in \C\,:\, |1 - w|^{1+T} |1+w|^{1-T} = (1+T)^{1+T} (1-T)^{1-T}\big\}.
\end{equation}

\begin{lemma}\label{sect3-lem1}
For any $T \in [0,1[$, $\Gamma_T$ is a lemniscate with double point $-T\in ]-1,0]$, such that the open subset
$$
\Omega_T = \Big\{w \in \C\,:\, \Big(\frac{|1-w|}{1+T}\Big)^{1+T}\, \Big(\frac{|1+w|}{1-T}\Big)^{1-T} < 1\Big\}
$$
is the union of two disjoint bounded domains $\Omega_T^-$ and $\Omega_T^+$ symmetric with respect to the real axis, which contain respectively
$-1$ and $1$. Moreover $\Omega_T^- \, \cap\, \R = ]-1-\rho\, (T), -T[$, where $\rho\, (T)$ is the unique root of
the strictly increasing function
$$
\rho \in ]0,+\infty[ \longmapsto \Big(\frac{\rho}{1-T}\Big)^{1-T}\,
\Big(\frac{\rho +2}{1+T}\Big)^{1+T} -1 \in ]-1,+\infty[
$$
and $T \in [0,1[ \longmapsto r(T) := 1 +\rho\, (T)$ is a strictly decreasing function from $\sqrt 2$ to $1_+$. Moreover $\Omega_T^+ \, \cap\, \R = ]-T, \check r\, (T)[$, where $T\in [0,1[ \longmapsto \check r\, (T)$ is a strictly increasing function from $\sqrt 2$ to $3_-$.
\end{lemma}

\begin{proof}
One has $-T \in \Gamma_T$. Let $w=-T + (1+T)\, w'= -T + (1+T)\, \gamma\, e^{i\theta}$. Then, if $c := (1+T)/(1-T)$,
\begin{equation*}
\begin{split}
\Gamma_T & := \Big\{w'\in \C\,:\, \big| 1 - w'\big|^{1+T} \, \big| 1 + c\, \, w'\big|^{1-T} = 1\Big\}\\
& = \Big\{(\gamma,\theta)\in \R^+ \times \R/(2\pi\Z)\,:\,
(\gamma^2 + 2 \gamma\, \cos \theta +1)^{1+T} (c^2\, \gamma^2 - 2c\, \gamma \cos \theta +
1)^{1-T} = 1\Big\}.
\end{split}
\end{equation*}
Since
\begin{equation*}
(1 - 2\cos \theta\, \gamma  +\gamma^2)^{1+T} (1 + 2c\cos \theta\, \gamma + c^2\, \gamma^2)^{1-T} -1 \\ =
2 \, c\, (1 - 2\, \cos^2\theta\big) \gamma^2 + O_\theta (\gamma^3),
\end{equation*}
the point $\{-T\}$ is a double point of $\Gamma_T$ and any ray $\{-T + \gamma\, e^{i\theta}\,:\, \gamma>0\}$ intersects $\Gamma_T$ transversally in exactly two points $\{\xi_{\theta}\}$ and $\{\check \xi_\theta\}$ respectively in $\{{\rm Re}\, w <-T\}$ and
$\{{\rm Re}\, w > -T\}$ provided that $1-2 \cos^2\, \theta < 0$, that is
$|\cos \theta| > 1/\sqrt 2 = \cos\, (\pi/4)$; otherwise it does not intersect $\Gamma_T$.
It follows that $\Gamma_T$ is a lemniscate with node $\{-T\}$ which is symmetric with respect to the real axis. Let us now consider the particular cases where $\theta=\pi$ and $\theta=0$. Since
$$
w \in [-1,-T[ \Longleftrightarrow  \frac{1-T}{1+T}\leq w' \leq 0,
$$
$\Gamma_T \,\cap\, [-1,-T[ = \emptyset$. A symmetric argument shows that $\Gamma_T \, \cap\, ]-T,1] = \emptyset$. If one lets $w =-1 -\rho$ with $\rho>0$ in \eqref{sect3-eq1}, one can see that $\Gamma_T \, \cap\, ]-\infty,-1[ =\{-1 - \rho\,(T)\} =  \{-r\,(T)\}$. One has that, for $T=0$, $\rho\, (0)\, (\rho\, (0)+2) = 1$, namely
$(\rho\, (0) +1)^2 = 2$, that is $r\, (0) = \sqrt 2$. Since for any $\rho>0$
$$
\lim\limits_{T \rightarrow 1_{-}}
\Big(\Big(\frac{\rho +2}{1+T}\Big)^{1+T} \exp \big ((1-T)(\log \rho - \log (1-T))\big) -1\Big)  =
\Big(1 + \frac{\rho}{2}\Big)^2 -1 > 0,
$$
one has $\lim_{T\rightarrow 1_{-}} (\rho\, (T)) =0_+$, hence $\lim_{T\rightarrow 1_{-}} (r\, (T)) = 1_{-}$.
Let $\rho>0$ and $T \in ]0,1[$. Then
$$
\frac{\partial}{\partial T}
\Big( (1-T) \, \log \frac{\rho}{1-T} + (1+T)\, \log \frac{\rho+2}{1+T}\Big)  = \log
\Big(\frac{\rho+2}{\rho} \, \frac{1-T}{1+T} \Big) \leq 0 \Longleftrightarrow
\rho \geq \frac{1}{T}-1.
$$
Since
$$
\Big( \frac{1/T -1 }{1 - T}\Big)^{1-T} \, \Big(\frac{(1/T-1)+2}{1 + T}\Big)^{1+T} -1 =
\frac{1}{T^2} - 1 >0 \Longrightarrow \rho\, (T) < \frac{1}{T} -1,
$$
one has
$$
\Big(\frac{\partial}{\partial T}
\Big( (1-T) \, \log \frac{\rho}{1-T} + (1+T)\, \log \frac{\rho+2}{1+T}\Big)\Big)_{\rho = \rho\, (T)} > 0.
$$
This implies that the function $T\in [0,1[ \rightarrow r\, (T)$ is strictly decreasing from $\sqrt 2$ to $1_+$.
Similarly, letting $w = 1+\rho$ with $\rho>0$ in \eqref{sect3-eq1} leads to $\Gamma_T \, \cap\, ]1,+\infty[ = \{1 + \check \rho\, (T)\}=\{\check r\, (T)\}$, where $\check \rho\, (T)$ is the unique root of the strictly increasing function
$$
\rho \in ]0,+\infty[ \longmapsto \Big(\frac{\rho}{1+T}\Big)^{1+T}\,
\Big(\frac{\rho +2}{1-T}\Big)^{1-T} - 1 \in ]-1,+\infty[.
$$
Since one has for any $\rho>2$ that
$$
\lim\limits_{T \rightarrow 1_{-}}
\Big(\Big(\frac{\rho}{1+T}\Big)^{1+T} \exp \big ((1-T)(\log (\rho +2) - \log (1-T))\big) -1\Big)  =
\Big(\frac{\rho}{2}\Big)^2 - 1 >0,
$$
a similar argument as that used to describe the behavior of $T\in [0,1[ \mapsto r(T)$ shows that
function $T\in [0,1[ \longmapsto \check r\, (T)$ is strictly increasing from $\sqrt{2}$ to
$3_{-}$. This completes the proof of Lemma \ref{sect3-lem1}.
\end{proof}

Let $m\geq 1$. For each non-empty ordered subset $J: 1 \leq j_1 < \cdots < j_\mu\leq m$ of $\{1,...,m\}$, we denote from now on by $\pi_J$ the projection $\bfw \in \C^m \longmapsto \bfw_J = (w_{j_1},...,w_{j_\mu}) \in \C^\mu$.

\begin{lemma}\label{sect3-lem2} Let $m\geq 1$,
$({\bf R},{\bf d})$ be a metric space and $\phi : (\bfa,x) \longmapsto \phi\, (\bfa, x)\in \C$ be a continuous function on $\big(\prod_{j=1}^m ]-r(T_j),r(T_j)[\big) \times {\bf R}$, with $T_j\in [0,1[$ and
$r(T_j) \in ]1,\sqrt 2]$ defined as in Lemma {\rm \ref{sect3-lem1}} for $j=1,...,m$. Let
$\bfr(\bfT) = (r(T_1),...,r(T_m))$. Suppose that for each pair of disjoint
subsets $\bfJ = (J^-,J^+)$ of $\{1,...,m\}$ with
$J= J^- \cup J^+ \not=\emptyset$ and $J':= \{1,...,m\}
\setminus J$, there is a continuous map
$$
f_{\bfJ}: (\bfalpha,x) \in \pi_{J'} ([-1,1]^m) \times {\bf R} \longmapsto f_{\bfJ,\bfalpha,x} \in H\big(\pi_J (\mathbb D^m_{\bfr(\bfT)})\big)
$$
such that for $x\in {\bf R}$ and
$\bfa \in \prod_{j=1}^m ]-r(T_j),r(T_j)[$
\begin{equation}\label{sect3-eq2}
\begin{cases} & a_j \leq -T_j\ \text{for}\ j\in J^- \\
& a_j \geq T_j\ \text{for}\ j \in J^+ \\
& a_j \in [-1,1]\ \text{for}\ j\in J'
\end{cases}
\Longrightarrow \phi\, (\bfa,x) = f_{\bfJ,\bfa_{J'},x} (\bfa_J).
\end{equation}
Then
\begin{equation}\label{sect3-eq3}
\phi\, (\bfa,x) = \lim\limits_{\{\bfN\in (\N^*)^m\,:\, \min \bfN \rightarrow +\infty\}}
\sum_{\bfnu \prec \bfN} \binom{\bfN}{\bfnu}
\, \Big(\frac{1-\bfa}{2}\Big)^{\bfnu} \, \Big(\frac{1+\bfa}{2}\Big)^{\bfN - \bfnu}\, \phi\, \Big({\bf 1} - 2\, \frac{\bfnu}{\bfN},x\Big)
\end{equation}
uniformly in $(\bfa,x)$ on any compact subset of $\big(\prod_{j=1}^m ]-r(T_j),r(T_j)[\big)\times
{\bf R}$.
\end{lemma}

\begin{proof}
It follows from \eqref{sect1-eq2bis} that \eqref{sect3-eq3} holds uniformly on $[-1,1]^m \times {\rm K}$, where ${\rm K}$ is any compact subset of ${\bf R}$. Since
\begin{equation*}
\prod_{j=1}^m \big]-r(T_j),r(T_j)\big[ =
\prod_{j=1}^m \Big(
\big]-r(T_j), -T_j\big[ \, \cup\,  \big[-1,1\big]\, \cup\, \big]T_j, r(T_j)\big[ \Big),
\end{equation*}
proving the lemma amounts to prove the following assertion ($(A_{m,m})$): given any pair of disjoint
subsets $\bfJ = (J^-,J^+)$ of $\{1,...,m\}$ with $J= J^- \cup J^+ \not=\emptyset$ and $J':= \{1,...,m\}
\setminus J$, then \eqref{sect3-eq3} holds on any compact set $K_{\bfJ} \times {\rm K}$, where ${\rm K}$ is a compact subset of ${\bf R}$ and $K_{\bfJ}$ is a compact subset of $\prod_{j=1}^m ]-r(T_j),r(T_j)[$ which is defined by
\begin{equation}\label{sect3-eq4}
\bfa \in K_{\bfJ} \Longleftrightarrow
\begin{cases} a_j \in [\gamma_j^-,\delta_j^-] \subset \Omega_{T_j}^-\ \text{when}\ j \in J^- \\
a_j \in [\gamma_j^+,\delta_j^+] \subset - \Omega_{T_j}^-\ \text{when}\ j \in J^+ \\
a_j \in [-1,1]\ \text{when}\ j \in J'
\end{cases}.
\end{equation}
In order to prove $(A_{m,m})$, we use an induction on $m$, then on the cardinal $1\leq \mu\leq m$ of $J = J^- \cup J^+$.
Let us start with $m$ arbitrary, $\mu=1$ and re-index coordinates so that $\bfJ = (\{m\},\emptyset)$ or
$(\emptyset,\{m\})$.
Given $\bfN\in (\N^*)^m$, $\bfa \in K_{\bfJ}$ and $x\in {\rm K}$, observe that
\begin{multline}\label{sect3-eq10}
\sum_{\bfnu \prec \bfN} \binom{\bfN}{\bfnu}
\, \Big(\frac{1-\bfa}{2}\Big)^{\bfnu} \, \Big(\frac{1+\bfa}{2}\Big)^{\bfN - \bfnu}\, \phi\, \Big({\bf 1} - 2\, \frac{\bfnu}{\bfN},x\Big) \\
= \sum\limits_{\nu_m=0}^{N_m}
\binom{N_m}{\nu_m} \Big(\frac{1-a_m}{2}\Big)^{\nu_m} \, \Big(\frac{1+a_m}{2}\Big)^{N_m - \nu_m}
\mathbb B_{\phi\, (\cdot,1-2\frac{\nu_m}{N_m},x)} (\bfa')
\end{multline}
where
\begin{multline}\label{sect3-eq10bis}
\mathbb B_{\phi(\cdot,1-2\frac{\nu_m}{N_m},x)} (\bfa')\\
= \sum_{\bfnu' \prec \bfN'} \binom{\bfN'}{\bfnu'}
\, \Big(\frac{1-\bfa'}{2}\Big)^{\bfnu} \, \Big(\frac{1+\bfa'}{2}\Big)^{\bfN' - \bfnu'}\, \phi\, \Big({\bf 1} - 2\, \frac{\bfnu'}{\bfN'}\,\,
1 - 2\, \frac{\nu_m}{N_m}\,,\, x\Big)
\end{multline}
(with $\bfN'= (N_1,...,N_{m-1})$, $\bfnu'=(\nu_1,...,\nu_{m-1})$, $\bfa'= (a_1,...,a_{m-1}) = \bfa_{J'}$).
Suppose that $\bfJ = (\{m\},\emptyset)$. Then \eqref{sect3-eq10} splits as
\begin{multline}\label{sect3-eq11}
\sum\limits_{\bfnu' \prec \bfN'} \binom{\bfN'}{\bfnu'}
\, \Big(\frac{1-\bfa'}{2}\Big)^{\bfnu'} \, \Big(\frac{1+\bfa'}{2}\Big)^{\bfN' - \bfnu'}\\
\Big(
\sum\limits_{\nu_m=0}^{N_m}
\binom{N_m}{\nu_m} \Big(\frac{1-a_m}{2}\Big)^{\nu_m}
\Big(\frac{1+a_m}{2}\Big)^{N_m - \nu_m}
\, f_{\bfJ,{\bf 1} -2\bfnu'/\bfN',x} \Big(1- 2 \, \frac{\nu_m}{N_m}\Big) +  \sum\limits_{\{\nu_m \prec N_m\,:\, 1-2\frac{\nu_m}{N_m} > -T\}} \\
\binom{N_m}{\nu_m}\, \Big(\frac{1-a_m}{2}\Big)^{\nu_m} \, \Big(\frac{1+a_m}{2}\Big)^{N_m - \nu_m}\,
\Big(\mathbb B_{\phi\, (\cdot,1-2\frac{\nu_m}{N_m},x)} (\bfa') - f_{\bfJ,{\bf 1} -2\frac{\bfnu'}{\bfN'},x} \Big(1- 2 \, \frac{\nu_m}{N_m}\Big)\Big)\Big)
\end{multline}
when $\bfJ = (\{m\},\emptyset)$ or
\begin{multline}\label{sect3-eq12}
\sum\limits_{\bfnu' \prec \bfN'} \binom{\bfN'}{\bfnu'}
\, \Big(\frac{1-\bfa'}{2}\Big)^{\bfnu'} \, \Big(\frac{1+\bfa'}{2}\Big)^{\bfN' - \bfnu'}\\
\Big(
\sum\limits_{\nu_m=0}^{N_m}
\binom{N_m}{\nu_m} \Big(\frac{1-a_m}{2}\Big)^{\nu_m}
\Big(\frac{1+a_m}{2}\Big)^{N_m - \nu_m}
\, f_{\bfJ,{\bf 1} -2\bfnu'/\bfN',x} \Big(1- 2 \, \frac{\nu_m}{N_m}\Big) +  \sum\limits_{\{\nu_m \prec N_m\,:\, 1-2\frac{\nu_m}{N_m} < T\}} \\
\binom{N_m}{\nu_m}\, \Big(\frac{1-a_m}{2}\Big)^{\nu_m} \, \Big(\frac{1+a_m}{2}\Big)^{N_m - \nu_m}\,
\Big(\mathbb B_{\phi\, (\cdot,1-2\frac{\nu_m}{N_m},x)} (\bfa') - f_{\bfJ,{\bf 1} -2\frac{\bfnu'}{\bfN'},x} \Big(1- 2 \, \frac{\nu_m}{N_m}\Big)\Big)\Big)
\end{multline}
when $\bfJ = (\emptyset,\{m\})$.
One has either
\begin{equation}\label{sect3-eq13}
\lim\limits_{N_m\rightarrow +\infty}
\sup_{\bfa\in K_{\bfJ}}\, \sum\limits_{\{\nu_m \prec N_m\,:\, 1-2\frac{\nu}{N} > -T\}}
\binom{N_m}{\nu_m}\, \Big|\frac{1-a_m}{2}\Big|^{\nu_m} \, \Big|\frac{1+a_m}{2}\Big|^{N_m - \nu_m} = 0
\end{equation}
when $\bfJ = (\{m\},\emptyset)$  or
\begin{equation}\label{sect3-eq14}
\lim\limits_{N\rightarrow +\infty}
\sup_{\bfa\in K_{\bfJ}}\, \sum\limits_{\{\nu \prec N\,:\, 1-2\frac{\nu}{N} < T\}}
\binom{N}{\nu}\, \Big|\frac{1-a}{2}\Big|^{\nu} \, \Big|\frac{1+a}{2}\Big|^{N - \nu} = 0
\end{equation}
when $\bfJ = (\emptyset,\{m\})$ according to Kantorovitch localization theorem \cite{Kanto31}, see also \cite[Theorem 4.1.3]{Lor53}, provided the identification of $[-1,1]$ with $[0,1]$ is realized through the affine correspondence $a\longleftrightarrow (a+1)/2$. Since
$$
\sup_{\bfalpha \in [-1,1]^m, x\in {\rm K}}
\big|B_{\phi(\cdot,\alpha_m,x)} (\bfalpha') - f_{\bfJ,\bfalpha',x} (\alpha_m)\big|
\leq \sup_{\bfalpha \in [-1,1]^m, x\in {\rm K}} \big(|\phi\, (\bfalpha,x)| + |f_{J,\bfalpha',x}(\alpha_m)|\big) < +\infty
$$
according to \eqref{sect3-eq10bis}, proving $(A_{m,1})$ amounts to prove, in view of \eqref{sect3-eq11} or
\eqref{sect3-eq12} that
\begin{multline*}
\lim\limits_{\min_j N_j \rightarrow +\infty}
\Big(\sum\limits_{\bfnu' \prec \bfN'} \binom{\bfN'}{\bfnu'}
\, \Big(\frac{1-\bfa'}{2}\Big)^{\bfnu'} \, \Big(\frac{1+\bfa'}{2}\Big)^{\bfN' - \bfnu'}\\
\sum\limits_{\nu_m=0}^{N_m}
\binom{N_m}{\nu_m} \Big(\frac{1-a_m}{2}\Big)^{\nu_m}
\Big(\frac{1+a_m}{2}\Big)^{N_m - \nu_m}
\, f_{\bfJ,{\bf 1} -2\bfnu'/\bfN',x} \Big(1- 2 \, \frac{\nu_m}{N_m}\Big)\Big)\Big)
= \phi\, (\bfa,x)
\end{multline*}
uniformly in $\bfa\in K_{\bfJ}$, such convergence being uniform with respect to $x\in {\rm K}$. Theorem \ref{sect2-thm1} with $m=1$ implies that
$$
\lim\limits_{N_m\rightarrow +\infty} \Big(\sum\limits_{\nu_m=0}^{N_m}
\binom{N_m}{\nu_m} \Big(\frac{1-a_m}{2}\Big)^{\nu_m}
\Big(\frac{1+a_m}{2}\Big)^{N_m - \nu_m}
\, f_{\bfJ,\bfalpha,x} \Big(1- 2 \, \frac{\nu_m}{N_m}\Big)\Big) =
f_{\bfJ,\bfalpha,x}\, (a_m)
$$
uniformly in $a_m\in \pi_{\{m\}} (K_{\bfJ})$, such convergence being uniform with respect to
the parameters $(\bfalpha,x) \in [-1,1]^{m-1}\times {\rm K}$. On the other hand, one has
\begin{multline*}
\lim\limits_{\min \bfN'\rightarrow +\infty}
\Big(\sum\limits_{\bfnu' \prec \bfN'} \binom{\bfN'}{\bfnu'}
\, \Big(\frac{1-\bfa'}{2}\Big)^{\bfnu'} \, \Big(\frac{1+\bfa'}{2}\Big)^{\bfN' - \bfnu'}
f_{\bfJ, {\bf 1} - 2\, \bfnu'/\bfN',x}\, (a_m)\Big)\Big) \\
 = f_{\bfJ,\bfa',x}(a_m) = \phi\, (\bfa,x)
\end{multline*}
uniformly in $\bfa'\in [-1,1]^{m-1}$, such convergence being uniform with respect to the parameters
$(a_m,x) \in \pi_{\{m\}}(K_{\bfJ})\times {\rm K}$, according to
\eqref{sect1-eq2bis} with $m$ replaced by $m-1$. This concludes the proof of assertion $(A_{m,1})$ for any $m\in \N^*$.
In particular, it concludes the proof of Lemma \ref{sect3-lem2} when $m=1$.
We take now $m\geq 2$ and assume that Lemma \ref{sect3-lem2} is proved up to the step $m-1$.
Let $\bfJ =(J^-,J^+)$, with $J^-,J^+ \in \{1,...,m\}$ be such that ${\rm card}\, J=\mu \geq 2$.
We may re-index coordinates in order that $j_\mu=m$ and $J'= \{1,....,m-\mu\}$. We repeat the previous argument but use instead the inductive hypothesis $(A_{m-1,\mu-1})$ in order to ensure that
$$
\sup_{\bfalpha'\in \pi_{\{1,...,m-1\}}(K_\bfJ),\alpha_m \in [-1,1],
x\in {\rm K}}
\big|B_{\phi(\cdot,\alpha_m,x)} (\bfalpha') - f_{\bfJ,\bfalpha',x} (\alpha_m)\big| < +\infty.
$$
Let $\bfN_J := (N_{m-\mu=1},...,N_m)$, $\bfN_{J'} := (N_1,...,N_{m-\mu})$. Theorem \ref{sect2-thm1} with $m$ replaced by $\mu$ implies that
\begin{multline*}
\lim\limits_{\min \bfN_J\rightarrow +\infty}
\Big(\sum\limits_{\bfnu_J \prec \bfN_J}
\binom{\bfN_J}{\bfnu_J} \Big(\frac{1-\bfa_J}{2}\Big)^{\bfnu_J}
\Big(\frac{1+\bfa_J}{2}\Big)^{\bfN_J - \bfnu_J}
\, f_{\bfJ,\bfalpha,x} \Big(1- 2 \, \frac{\bfnu_J}{\bfN_J}\Big)\Big) \\
= f_{\bfJ,\bfalpha,x} (a_\bfJ)
\end{multline*}
uniformly in $\bfa_J \in \pi_J (K_{\bfJ})$, such convergence being uniform with respect to the parameters
$(\bfalpha,x) \in [-1,1]^{m-\mu} \times {\rm K}$. When $\mu=m$, we are done. In case $\mu<m$, one has also
\begin{multline*}
\lim\limits_{\min \bfN_{J'}\rightarrow +\infty}
\Big(\sum\limits_{\bfnu_{J'} \prec \bfN_{J'}} \binom{\bfN_{J'}}{\bfnu_{J'}}
\, \Big(\frac{1-\bfa_{J'}}{2}\Big)^{\bfnu_{J'}} \, \Big(\frac{1+\bfa_{J'}}{2}\Big)^{\bfN_{J'} - \bfnu_{J'}}
f_{\bfJ, {\bf 1} - 2\, \frac{\bfnu_{J'}}{\bfN_{J'}},x}\, (a_J)\Big)\Big) \\
 = f_{\bfJ,\bfa_{J'},x}(\bfa_J) = \phi\, (\bfa,x)
\end{multline*}
uniformly in $\bfa_{J'} \in [-1,1]^{m-\mu}$, such convergence being uniform with respect to the parameters
$(\bfa_J,x) \in \pi_J(K_{\bfJ}) \times {\rm K}$, according to \eqref{sect1-eq2bis} with $m$ replaced by
$m-\mu$. We are then done again. This concludes the proof of the assertion $(A_{m,\mu})$, hence of Lemma \ref{sect3-lem2}.
\end{proof}

In the global context of $\R^m$ instead of $\prod_{j=1}^m ]-r(T_j),r(T_j)[$ with $\bfT\in ([0,1[)^m$, we introduce the following concept of {\it $\bfT$-predictability}. Observe that the hypothesis in Theorem \ref{sect2-thm1} and Theorem \ref{sect2-thm2} with $r_j=+\infty$ for any $j$ correspond to {\it ${\bf 0}_-$-predictability}, namely $-1 << T_j< 0$ for any $j=1,...,m$.

\begin{definition}\label{sect3-def1}
Let $m\in \N^*$, $({\bf R},{\bf d})$ be a metric space and $\bfT \in [0,1[^m$.
A continuous complex valued function
$\phi~: (\bfa,x) \in \R^m \times {\bf R}\longmapsto \phi\, (\bfa,x)\in \C$ is said to be
{\it $\bfT$-predictable}  if and only if for each pair of disjoint
subsets $\bfJ = (J^-,J^+)$ of $\{1,...,m\}$ with
$J= J^-\, \cup\, J^+ \not=\emptyset$ and $J':= \{1,...,m\}
\setminus J$, there is a continuous map
$$
F_{\bfJ}: (\bfalpha,x) \in \pi_{J'} (\R^m) \times {\bf R} \longmapsto F_{\bfJ,\bfalpha,x} \in H\big(\pi_J (\C^m)\big)
$$
such that for $(\bfa, x)\in \R^m\times {\bf R}$
\begin{equation}\label{sect3-eq18}
\begin{cases} & a_j \leq -T_j\ \text{for}\ j\in J^- \\
& a_j \geq T_j\ \text{for}\ j \in J^+ \\
\end{cases}
\Longrightarrow \phi\, (\bfa,x) = F_{\bfJ,\bfa_{J'},x} (\bfa_J),
\end{equation}
where $\bfa_J = \pi_J (\bfa)$ and $\bfa_{J'} = \pi_{J'} (\bfa)$.
\end{definition}

\begin{example}\label{sect3-expl1} Let $x\in {\bf R} \longmapsto F_x$ be a continuous map from
${\bf R}$ to $H(\C^m)$. The function
$$
\phi~: (\bfa,x) \in \R^m \times {\bf R} \longmapsto F_x (|a_1|,...,|a_m|)
$$
is ${\bf 0}$-predictable with $F_{J,\bfa_{J'},x} (\bfw) = F_x\, (\check \bfw)$, where $\check w_j = -w_j$ for $j\in J^-$, $\check w_j= w_j$ for $j\in J^+$ and one specifies $\check w_{j'} = |a_{j'}|$ for $j\in J'$.
\end{example}

\begin{theorem}\label{sect3-thm1}
Let $\bfT \in [0,1[^m$ and $\phi : (\bfa,x) \in \R^m \times {\bf R}$ be a $\bfT$-predictable continuous complex valued function.
Then, for or any $\bfrho =(\rho_1,...,\rho_m)$ such that $1<\rho_j<r(T_j)$ for $j=1,...,m$ and for any $k\in \N$
\begin{multline}\label{sect3-eq19}
\phi\, (\bfa,x)
= \lim_{\{\bfN_0\in (\N^*)^m\,:\, \min \bfN_{0} \rightarrow +\infty\}} \bigg(
\cdots \bigg(\lim_{\{\bfN_{k} \in (\N^*)^m\,:\, \min \bfN_k \rightarrow +\infty\}} \bigg(
\sum_{\bfnu_0 \prec \bfN_0} \\
\bigg(
\sum_{\bfnu_k\prec \bfN_k} \binom{\bfN_k}{\bfnu_k}
\, \Big(\frac{1-\bfrho^{-k}\cdot \bfa}{2}\Big)^{\bfnu_k} \, \Big(\frac{1+\bfrho^{-k}\cdot \bfa}{2}\Big)^{\bfN_k - \bfnu_k}
\bigg( \sum_{\bfnu_{k-1}\prec \bfN_{k-1}} \cdots \sum_{\bfnu_{1}\prec \bfN_{1}}
\\
\prod_{\kappa = 0}^{k-1} \Big(
\binom{\bfN_{\kappa}}{\bfnu_{\kappa}}
\Big(\frac{1}{2}
\Big(1 - \bfrho \cdot \Big({\bf 1} - 2\, \frac{\bfnu_{\kappa+1}}{\bfN_{\kappa+1}}\Big)\Big)\Big)^{\bfnu_{\kappa}} \Big(\frac{1}{2}
\Big(1 + \bfrho \cdot \Big({\bf 1} - 2\, \frac{\bfnu_{\kappa+1}}{\bfN_{\kappa+1}}\Big)\Big)\Big)^{\bfN_{\kappa}- \bfnu_{\kappa}}\Big)\Big)\bigg)\bigg) \\
\phi\, \Big(1 - 2\, \frac{\bfnu_{0}}{\bfN_{0}},x\Big)\bigg) \bigg) \cdots \bigg)
\end{multline}
uniformly on any compact subset of $\big(\prod_{j=1}^m \big]-\rho_j^k\, r(T_j),\rho_j^k\, r(T_j)\big[\big)\times {\bf R}$, where one denotes $\bfrho^k\cdot \bfalpha := (\rho_1^k\, \alpha_1,...,\rho_m^k\, \alpha_m)$ for any
$\bfalpha\in \R^m$ and $k\in \Z$.
\end{theorem}

\begin{proof}
For each $k\in \N$, let $\phi^{[k]} : \R^m \times {\bf R} \longrightarrow \C$ be the continuous function
defined by $\phi^{[k]}\, (\bfa,x) := \phi\, (\bfrho^k \cdot \bfa,x)$.
For any partition $\{1,...,m\} = J\,\cup\, J'$ of $\{1,...,m\}$ and any $\bfa\in \R^m$, let also
$\bfa^{[k]}_J = \pi_J (\bfrho^k \cdot \bfa)$ and $\bfa^{[k]}_{J'}=\pi_{J'}(\bfrho^k\cdot \bfa)$. For any pair
$\bfJ = (J^-,J^+)$ of disjoint subsets of $\{1,...,m\}$ with $J^- \cup J^+ \not=\emptyset$ and any
$\bfa \in \R^m$ such that $a_j\leq -T_j \leq - T_j/\rho_j^k$ for $j\in J^-$,
$a_j\geq T_j \geq T_j/\rho_j^k$ for $j\in J^+$ and any $x\in {\bf R}$ and
$|a_j|\leq 1$ for $j\in J'$, one has
\begin{equation}\label{sect3-eq20}
\phi^{[k]}\, (\bfa, x) = F_{\bfJ, \bfa^{[k]}_{J'},x} (\bfa^{[k]}_J).
\end{equation}
It follows then from Lemma \ref{sect3-lem2} that
\begin{equation}\label{sect3-eq21}
\phi^{[k]}\, (\bfa,x) = \lim\limits_{\min_j N_j \rightarrow +\infty}
\sum_{\bfnu \prec \bfN} \binom{\bfN}{\bfnu}
\, \Big(\frac{1-\bfa}{2}\Big)^{\bfnu} \, \Big(\frac{1+\bfa}{2}\Big)^{\bfN - \bfnu}\, \phi^{[k]}\, \Big({\bf 1} - 2\, \frac{\bfnu}{\bfN},x\Big)
\end{equation}
uniformly in $(\bfa,x)$ on any compact subset of $\big(\prod_{j=1}^m ]-r(T_j),r(T_j)[\big) \times {\bf R}$.   When $k\in \N^*$, \eqref{sect3-eq21} can be reformulated as the {\it Bernstein $\bfrho$-pseudo-autosimilarity relation}
\begin{multline}\label{sect3-eq22}
\phi\, (\bfa,x) \\
= \lim\limits_{\min \bfN_k\rightarrow +\infty}
\sum_{\bfnu_k \prec \bfN_k} \binom{\bfN_k}{\bfnu_k}
\, \Big(\frac{1-\bfrho^{-k}\cdot \bfa}{2}\Big)^{\bfnu_k} \, \Big(\frac{1+\bfrho^{-k}\cdot \bfa}{2}\Big)^{\bfN_k - \bfnu_k}\, \phi \Big(\bfrho^k\cdot \Big({\bf 1} - 2\, \frac{\bfnu_k}{\bfN_k}\Big),x\Big)
\end{multline}
uniformly in $(\bfa,x)$ on any compact subset of $\big(\prod_{j=1}^m ]-\rho_j^k r(T_j),\rho_j^k r(T_j)[\big) \times {\bf R}$. Since $\bfrho\cdot [-1,1]^m$ is a compact subset of
$\prod_{j=1}^m ]-r(T_j),r(T_j)[$, $\bfrho^k \cdot |-1,1]^m$ is a compact subset of
$\prod_{j=1}^m ]-\rho_j^{k-1} r(T_j), \rho_j^{k-1} r(T_j)[$. When $k=1$, Lemma \ref{sect3-lem2} implies than that for any $\bfnu_1 \prec \bfN_1$
\begin{multline*}
\phi\, \Big(\bfrho \cdot \Big({\bf 1} - 2\, \frac{\bfnu_1}{\bfN_1}\Big),x\Big)
= \lim\limits_{\min \bfN_0\rightarrow +\infty} \\
\sum_{\bfnu_0 \prec \bfN_0} \binom{\bfN_0}{\bfnu_0}
\Big(\frac{1}{2}
\Big(1 - \bfrho\cdot \Big({\bf 1} - 2\, \frac{\bfnu_1}{\bfN_1}\Big)\Big)\Big)^{\bfnu_0} \Big(\frac{1}{2}
\Big(1 + \bfrho\cdot \Big({\bf 1} - 2\, \frac{\bfnu_1}{\bfN_1}\Big)\Big)\Big)^{\bfN_0- \bfnu_0}
\phi\, \Big( 1 - 2\, \frac{\bfnu_0}{\bfN_0},x\Big)
\end{multline*}
uniformly with respect to $\bfnu_1 \prec \bfN_1$ and $x$ in any compact subset of ${\bf R}$.
When $k\geq 2$, one has that for each $\bfnu_k\prec \bfN_k$ and each compact subset
${\rm K}$ of ${\bf R}$,
\begin{multline}\label{sect3-eq23}
\phi\, \Big(\bfrho^k\cdot \Big({\bf 1} - 2\, \frac{\bfnu_k}{\bfN_k}\Big),x\Big)
= \lim\limits_{\min \bfN_{k-1}\rightarrow +\infty} \\
\sum_{\bfnu_{k-1} \prec \bfN_{k-1}} \binom{\bfN_{k-1}}{\bfnu_{k-1}}
\Big(\frac{1}{2}
\Big(1 - \bfrho\cdot \Big({\bf 1} - 2\, \frac{\bfnu_k}{\bfN_k}\Big)\Big)\Big)^{\bfnu_{k-1}} \Big(\frac{1}{2}
\Big(1 + \bfrho \cdot \Big({\bf 1} - 2\, \frac{\bfnu_k}{\bfN_k}\Big)\Big)\Big)^{\bfN_{k-1}- \bfnu_{k-1}}\\
\phi\, \Big(\bfrho^{k-1}\cdot \Big(1 - 2\, \frac{\bfnu_{k-1}}{\bfN_{k-1}}\Big),x\Big)
\end{multline}
uniformly with respect to $\bfnu_k\prec \bfN_k$ and $x\in {\rm K}$.
The asymptotic formula \eqref{sect3-eq19} follows then inductively on $k$.
\end{proof}

\section{Examples related to harmonic analysis}\label{sect4}

Examples related to harmonic analysis which illustrate the results in \textsection \ref{sect2} or  \textsection \ref{sect3} are concerned with the case where
${\bf R}= \R^d\subset \C^d$, so that one denotes $z=(z_1,...,z_d)$ instead of $x$ the parameter in Theorems \ref{sect2-thm1} or \ref{sect2-thm2} (both with $r_j=+\infty$ for $j=1,...,m$).
The function $\phi$ is, from now on, the restriction to $\R^m \times {\bf R} \subset \R^m \times \C^d$ of an entire function $\Phi\in H(\C^m \times \C^d)$. We consider in particular the cases where $m=d,2d,3d$, with
\begin{equation}\label{sect4-eq1}
\begin{split}
\phi~&: (\bfomega,z) \in \R^d \times \C^d \longmapsto  G_1\, (i\, \bfomega \cdot z) := G_1\, (i\, \omega_1\, z_1,...,i\, \omega_d\, z_d) \\
\phi~&: (\bftau,z) \in \R^d \times \C^d \longmapsto  G_2\, (z - \bftau)\, G_1\, (i \,\bfomega \cdot z)
:= G_2\, (z_1 - \tau_1,...,z_d - \tau_d)\, G_1\, (i\, \bfomega \cdot z)
\\
\phi~&: (\bftau,\bfomega, z) \in \R^{2d} \times \C^d  \longmapsto G_2\, (z - \bftau/2) \, \overline{G_2\, (\overline z + \bftau/2)}\, G_1 (i\, \bfomega\cdot \bftau)\\
\phi~&: (\bftau, \bfalpha, z) \in \R^{2d} \times \C^d  \longmapsto 2^{-\bfalpha/2}\, G_3\Big((z -\bftau)\cdot \frac{1}{2^\bfalpha}\Big) := 2^{-\bfalpha/2}\, G_3\Big( \frac{z_1 - \tau_1}{2^{\alpha_1}},...,\frac{z_d-\tau_d}{2^{\alpha_d}}\Big) \\
\phi~&: (\bftau,\bfalpha,\bfomega, z) \in \R^{3d} \times \C^d  \longmapsto 2^{-\bfalpha/2}\, G_3\Big((z -\bftau)\cdot \frac{1}{2^\bfalpha}\Big)\, G_1(\bfomega \cdot z),
\end{split}
\end{equation}
where $G_1,G_2,G_3\in H(\C^d)$. The set of variables $\bfomega, \bftau, 2^{\bfalpha}$ vary respectively in the frequency domain $(\R^d)^\star\simeq \R^d$, the space domain $\R^d$, and the scale domain
$]0,+\infty[^d$ for $d$-variate complex valued signals. The different cases in \eqref{sect4-eq1} correspond then to the extension from $\R^d$ to $\C^d$ of elementary atoms respectively in Fourier analysis, Gabor analysis (when $(G_2)_{|\R^d}$ is a $L^1$-normalized centered gaussian atom), Fourier-Wigner-Ville analysis ($(G_2)_{|\R^d}$ being a superposition of tensorized gaussian chirps), time-scale analysis, time-scale-frequency analysis, both  with respect to a $L^2$-normalized wavelet $\Psi= (G_3)_{|\R^d}$ with bounded spectrum. \\

One considers as first illustration the (univariate) Schr\"odinger Cauchy problem for the free-particle
\begin{equation}\label{SECT5-eq0}
(\mathscr S_0) : \Big( i\, \frac{\partial}{\partial t} + \frac{\partial^2}{\partial x^2}\Big)\, \psi(t,x) = 0\quad (t>0,\ x\in \R) \,;\, \psi (0,x)= \varphi(x)
\end{equation}
The metric space of parameters ${\bf R}$ is for the moment either ${\bf R}= \R$ or ${\bf R} = ]0,\infty[ \times \R$, both equipped with their usual distance. We denote them respectively as $\boldsymbol R$ (with current point $x$) and $\boldsymbol {\mathcal R}$ (with current point $(t,x)$).
We recall that the fundamental solution for the Schr\"odinger operator \eqref{SECT5-eq0} in $\mathcal D'(\overline{\boldsymbol {\mathcal R}})$ is
$$
(t,x) \longmapsto \Big[
\frac{1}{\sqrt {4i\pi t}}\, \exp \Big( i\, \frac{x^2}{4t}\Big)\Big] = [\mathcal G_0(t,x)].
$$
The following two elementary lemmas describe the evolution under $\mathscr S_0$ respectively of univariate
chirps ($d=1$ and $\phi~: (\tau,\omega,z) \longmapsto \exp (i (\omega + (\tau/2) \, z)\, z)$) and univariate gaussian chirps ($d=1$ and $\phi~: (\tau_1,\tau_2,\alpha,\omega, z) \longmapsto
\exp (i (\omega + (\tau_1/2)\, z)\, z - (z-\tau_2)^2/2^\alpha)$). As functions of the complex variable $z$, observe that such functions $\phi$ belong to
$$
A_2(\C) = \{F \in H(\C)\,:\, |F(z)| = O(\exp (B|z|^2)\ \text{for some}\ B>0\},
$$
thus lie beyond
$$
A_1(\C) = \{F \in H(\C)\,:\, |F(z)| = O(\exp (B|z|)\ \text{for some}\ B>0\},
$$
which is the continuity domain for the Fresnel type operator given by convolution with $\mathscr G_0$ considered for example in \cite[\textsection 5.1]{CSSY22} or \cite{ABCS22}.

\begin{lemma}[evolution of chirps]\label{SECT5-lem1}
For any $(\tau,\omega)\in \R^2$ and any $(t,x)\in \boldsymbol {\mathcal R}$ such that
$2 \tau t \not=-1$,
\begin{equation}\label{SECT5-eq1}
\int_\R \exp \Big( i\, \Big(\Big( \omega +\frac{\tau}{2}\, y\Big)\, y\Big)\Big)\, \mathcal G_0 (t,x-y)\, dy
=
\frac{1}{\sqrt{\nu_\tau(t)}}\,
\exp \Big (i\, \frac{1}{\nu_\tau(t)} \Big( \Big( \omega + \frac{\tau}{2}\, x\Big)\, x - \omega^2\, t \Big),
\end{equation}
where $\nu_\tau(t):= 1 + 2\tau t$ and $\sqrt {-\xi} := i \sqrt{\xi}$ for $\xi > 0$.
One defines in this way a distribution $\psi_{\omega}^\tau \in L^\infty_{\rm loc} (\boldsymbol {\mathcal R})$ with singular support empty when $\tau \geq 0$ or equal to $\{(t,x)\in \boldsymbol {\mathcal R}\,:\, \nu_\tau(t)= 0\}$ when $\tau <0$.
\end{lemma}

\begin{proof}
For any $\tau \in \R$ and $t>0$, let
$$
\lambda(t)=\lambda_\tau (t) := \frac{\tau}{2} + \frac{1}{4t} = \frac{1 + 2\tau t}{4t}.
$$
For any $(t,x)\in \boldsymbol {\mathcal R}$, one has then
\begin{multline}\label{SECT5-eq2}
\frac{1}{\sqrt{4\pi i t}} \int_{\R}
\exp \Big( i\, \Big(\Big( \omega +\frac{\tau}{2}\, y\Big)\, y + \frac{(x-y)^2}{4t}\Big)\Big)\, dy \\
= \frac{1}{\sqrt{4\pi it}}\, \exp \Big( i\, \frac{x^2}{4t}\Big)\, \int_\R
\exp \Big( i\, \Big( \lambda (t)\, y^2 - y\, \Big( \frac{x}{2t} -\omega \Big)\Big)\, dy.
\end{multline}
When $\lambda (t)>0$,
\begin{multline}\label{SECT5-eq3}
\frac{1}{\sqrt{4\pi i t}}\, \int_\R
\exp \Big( i\, \Big( \lambda (t)\, y^2 - y\, \Big( \frac{x}{2t} -\omega \Big)\Big)\, dy \\
= \frac{1}{\sqrt{\lambda (t)}}\, \Big(\frac{1}{\sqrt{4\pi i t}}\,
\int_\R
\exp \Big( i\, \Big(v - \frac{1}{2\sqrt{\lambda (t)}}\, \Big(\frac{x}{2t} -\omega \Big) \Big)^2\Big)\, dv\Big)\,
\exp \Big( -\frac{i}{4 \lambda (t)}\, \Big(\frac{x}{2t} - \omega \Big)^2\Big) \\
= \frac{1}{\sqrt{\lambda (t)}}
\, \exp \Big( - \frac{i}{4t}\, \frac{1}{1+2\tau t} \, (x-2\omega t)^2\Big).
\end{multline}
When $\lambda (t)<0$,
\begin{multline}\label{SECT5-eq3bis}
\frac{1}{\sqrt{4\pi i t}}\, \int_\R
\exp \Big( i\, \Big( \lambda (t)\, y^2 - y\, \Big( \frac{x}{2t} -\omega \Big)\Big)\, dy= \\
\frac{1}{\sqrt{|\lambda (t)|}}\,\Big(\overline{ \frac{1}{\sqrt{- 4\pi i t}}\, \int_\R
\exp \Big( i\, \Big(v  + \frac{1}{2 \sqrt{|\lambda(t)|}}\, \Big( \frac{x}{2t} -\omega \Big)\Big)^2\Big)\, dy }\Big)\,
\exp\Big( \frac{i}{4t} \frac{1}{|1 + 2\tau t|}\, (x-2\omega t)^2\Big)\\
= \frac{1}{\sqrt{\lambda (t)}}\, \exp\Big( - \frac{i}{4t} \frac{1}{1 + 2\tau t}\, (x-2\omega t)^2\Big).
\end{multline}
Lemma \ref{SECT5-lem1} follows from \eqref{SECT5-eq2} combined with either \eqref{SECT5-eq3} or \eqref{SECT5-eq3bis}.
\end{proof}

\begin{lemma}[evolution of gaussian chirps]\label{SECT5-lem2}
For any $(\tau_1,\tau_2,\alpha)\in \R^3$, $\omega \in \R$ and any $(t,x)\in \boldsymbol {\mathcal R}$,
\begin{multline}\label{SECT5-eq5}
\int_\R \exp \Big(
i\, \Big( \omega + \frac{\tau_1}{2}\, y\Big)\, y - \frac{(y-\tau_2)^2}{2^\alpha}\Big)\, \mathcal G_0(t, x-y)\, dy \\
= \frac{\gamma_{\bftau,\alpha} (t)}{\sqrt{\nu_{\tau_1,\alpha}(t)}}\, \exp
\Big(\frac{1}{\nu_{\tau_1,\alpha} (t)}\, \Big(i\Big( \omega +
\frac{\tau_1}{2}\, x\Big)\, x   - \frac{(x-\tau_2)^2}{2^\alpha} - i\, (\omega + \tau_1 \tau_2)^2\, t\Big),
\end{multline}
where $\nu_{\tau_1,\alpha}(t) =
1 + 2\tau_1 t + i\, 2^{2-\alpha} t$ and
\begin{equation*}
\gamma_{\bftau,\alpha} (t) = \exp \Big(\frac{i}{\nu_{\tau_1,\alpha}(t)}
(\nu_{\tau_1,\alpha} (t)-1) \, \tau_2\, (\omega +\tau_1\tau_2\)\Big).
\end{equation*}
\end{lemma}

\begin{proof}
Let
$$
\mu(t) = \mu_{\bftau,\alpha}(t):=  -i\, \frac{1 + 2 \tau_1 t + i\, 2^{2-\alpha}\, t}{4t}.
$$
One has for any $(t,x)\in \boldsymbol {\mathcal R}$ that
\begin{multline}\label{SECT5-eq5-1}
\frac{1}{\sqrt{4\pi i t}} \int_{\R}
\exp \Big(- \frac{(y-\tau_2)^2}{2^\alpha} +
i\, \Big( \Big( \omega + \frac{\tau_1}{2} y \Big) y + \frac{(x-y)^2}{4t}\Big)\Big)\Big)dy \\
= \frac{1}{\sqrt{4\pi i t}}
\, \exp\Big( i \Big( \tau_2 \Big( \omega + \frac{\tau_1 \tau_2}{2}\Big) +
\frac{(x-\tau_2)^2}{4t}\Big)\Big)
\int_\R \exp\Big( - \mu (t) y^2 + iy \, \Big(\omega - \frac{x- \nu_{\tau_1}(t) \tau_2}{2t}\Big)\Big) dy.
\end{multline}
It follows from the analytic continuation principle in
the half space $\{\lambda\in \C\,:\,
{\rm Re}\, \lambda >0\}$ that for any $(t,x)\in \boldsymbol {\mathcal R}$
\begin{multline}\label{SECT5-eq5-2}
\frac{1}{\sqrt{4\pi i t}}\int_\R \exp\Big( - \mu (t)\, y^2 + iy \, \Big(\omega - \frac{x- \nu_{\tau_1} (t) \tau_2}{2t}\Big)\Big)\, dy \\
= \frac{1}{\sqrt{8i\pi t \mu(t)}}\, \int_\R \exp\Big( - \frac{v^2}{2} + i \frac{v}{\sqrt{2
\mu (t)}}\, \Big(\omega - \frac{x- \nu_{\tau_1}(t) \tau_2}{2t}\Big)\Big)\, dv \\
= \frac{1}{\sqrt{\nu_{\tau_1,\alpha}(t)}}\,
\exp \Big( - \frac{1}{4 \mu (t)}\, \Big(\frac{x- \tau_2}{2t} - \omega - \tau_1 \tau_2\Big)^2\Big).
\end{multline}
One has
\begin{equation}\label{SECT5-eq5-3}
\begin{split}
\frac{(x-\tau_2)^2}{4t} \Big(i - \frac{1}{4 t\, \mu(t)}\Big) & = \frac{1}{\nu_{\tau_1,\alpha}(t)}
\,
\Big( i\, \frac{\tau_1}{2} - \frac{1}{2^\alpha}\Big)\, (x-\tau_2)^2 \\
& = \frac{1}{\nu_{\tau_1,\alpha}(t)}
\Big( i\frac{\tau_1}{2}\, x^2 - \frac{(x-\tau_2)^2}{2^\alpha} + i\, \frac{\tau_1 \tau_2^2}{2} - i\, \tau_1 \tau_2 \, x\Big) \\
\exp \Big(i\tau_2 \, \omega + \frac{1}{4 t \mu (t)}
\, (x-\tau_2)\, \omega \Big) & = \exp \Big(i \tau_2 \omega\, \frac{\nu_{\tau_1,\alpha}(t)-1}
{\nu_{\tau_1,\alpha}(t)} +i\, \frac{\omega}{\nu_{\tau_1,\alpha}(t)}\, x\Big) \\
\exp \Big( i\frac {\tau_1 \tau_2^2}{2} + \frac{\tau_1 \tau_2}{4 t
\mu(t)} (x-\tau_2)\Big) & =
\exp \Big( i\Big( \frac{\tau_1 \tau_2^2}{2}\Big(1 -
\frac{2}
{\nu_{\tau_1,\alpha}(t)}\Big) + \frac{\tau_1 \tau_2}{\nu_{\tau_1,\alpha}(t)}\, x\Big) \Big).
\end{split}
\end{equation}
One substitutes the equality \eqref{SECT5-eq5-2} in the right-hand side of \eqref{SECT5-eq5-1}, then the relations \eqref{SECT5-eq5-3} in the equality thus obtained.
\end{proof}

Lemma \ref{SECT5-lem3} concerns the evolution (in terms of the frequency indicator $\omega$) of $x\mapsto \phi(\omega,x)$ when
$\phi(\omega,z) = \exp (i\, |\omega z|)$ (observe that such function $\phi$ fails to be real analytic in $z$ on the real line, hence do not belong to $A_1(\C)$).

\begin{lemma}\label{SECT5-lem3}
For any $\omega \in \R$ and $(t,x) \in \boldsymbol {\mathcal R}$, one has
\begin{multline}\label{SECT5-eq6}
\int_\R \exp (i\, |\omega|\, |y|)\, \mathcal G_0 (t, x-y)\, dy =
\exp (-i \omega^2 \, t)\, \Big( \cos (|\omega|\, x) \\
+
\exp (-i |\omega| x)
\int_0^{-x - 2 t |\omega|} \exp \Big( i \frac{z^2}{4t}\Big)\, dz
- \exp (i |\omega| x) \, \int_0^{-x + 2t|\omega|} \exp \Big( i \frac{z^2}{4t}\Big)\, dz\Big).
\end{multline}
\end{lemma}

\begin{proof}
One has for any $\omega \in \R$ and $(t,x)\in \boldsymbol {\mathcal R}$,
\begin{equation*}
\begin{split}
& \sqrt{4i\pi t}\, \int_\R \exp (i\, |\omega|\, |y|)\, \mathcal G_0 (t, x-y)\, dy \\
& = \exp (-i |\omega|\, x)\, \int_{-\infty}^{-x}
\exp \Big( i\, \Big(\frac{z^2}{4t} - |\omega|\, z\Big)\Big)\, dz + \exp (i |\omega|\, x)\, \int_{-x}^{+\infty}
\exp \Big( i\, \Big(\frac{z^2}{4t} + |\omega|\, z\Big)\Big)\, dz\\
& =
\exp (-i \omega^2 t)
\Big( \exp (-i |\omega|\, x)\,
\int_{-\infty}^{-x - 2t |\omega|} \exp \Big( i
\frac{z^2}{4t}\Big) dz \, +\,   \exp (i |\omega| x)\,
\int_{-x + 2 t |\omega|}^\infty  \exp \Big( i
\frac{z^2}{4t}\Big) dz\Big).
\end{split}
\end{equation*}
Formula \eqref{SECT5-eq6} follows then from the fact that for any $t>0$
$$
\frac{1}{\sqrt{4i\pi t}}\, \int_{-\infty}^0 \exp \Big( i
\frac{z^2}{4t}\Big)\, dz = \frac{1}{\sqrt{4i\pi t}}\, \int_{0}^\infty \exp \Big( i
\frac{z^2}{4t}\Big)\, dz = \frac{1}{2}.
$$
\end{proof}

Next
Lemma \ref{SECT5-lem4}
describe the evolution of $x \mapsto \phi(\tau,\alpha,\omega,x)$, where $\phi$ corresponds to the final situation listed in \eqref{sect4-eq1}.

\begin{lemma}[evolution of time-scale-frequency atoms]\label{SECT5-lem4}
Let $\Psi~: x \longmapsto \Psi(x)$ be the spectrum of a compactly supported integrable signal.
For any $(\tau,\alpha,\omega)\in \R^3$ and $(t,x) \in \boldsymbol{\mathcal R}$,
\begin{multline}\label{SECT5-eq7}
\int_\R
\Psi \Big( \frac{y - \tau}{2^\alpha}\Big)\, \exp (i\, \omega\, y)\, \mathcal G_0(t, x-y)\, dy
\\
= \frac{1}{2\pi}
\int_\R \widehat \Psi(\xi) \exp
\Big( i \Big( - \frac{\tau\, \xi}{2^\alpha} + \Big(\omega + \frac{\xi }{2^\alpha}\Big)\, x -
\Big( \omega + \frac{\xi}{2^\alpha}\Big)^2 t\Big)\Big)\, d\xi.
\end{multline}
\end{lemma}

\begin{proof}
Fourier inversion formula implies that for any $y,\omega, \tau,\alpha \in \R$,
$$
\Psi \Big( \frac{x-\tau}{2^\alpha}\Big)\, \exp (i \, \omega\,  y)  = \frac{1}{2\pi} \int_{{\rm Supp}\, \widehat \Psi} \widehat \Psi (\xi) \, \exp \Big( - i
\frac{\tau \xi}{2^\alpha}\Big)\, \exp \Big( i \Big(\omega + \frac{\xi}{2^\alpha}\Big) y\Big)\, d\xi.
$$
Since
$$
\int_{\R}  \exp \Big( i \Big(\omega + \frac{\xi}{2^\alpha}\Big) y\Big)\, \mathcal G_0 (t,x-y)\, dy
= \exp \Big( i \Big( \Big(\omega + \frac{\xi}{2^\alpha}\Big)\, x - i
\Big(\omega + \frac{\xi}{2^\alpha}\Big)^2\, t\Big)\Big)
$$
for any $\xi \in {\rm Supp} \, \widehat \Psi$, $\alpha \in \R$ and $(t,x)\in \boldsymbol {\mathcal R}$, formula \eqref{SECT5-eq7} follows from the continuity properties of Fresnel type operators on ${\rm Exp}\, (\C) = A_1(\C)$, see \cite[\textsection 5.1]{CSSY22} or also \cite{ABCS22}.
\end{proof}

The next two propositions are immediate consequences of Theorem \ref{sect2-thm1}.

\begin{proposition}\label{SECT5-prop1} Let $\alpha_0 \in \R$ and $(t,x)\in \boldsymbol{\mathcal R} \longmapsto \psi_{\alpha_0,\tau,\omega} (t,x)$ be the evolution under
$\mathscr S_0$ of the translated modulated gaussian atom
$$
x \longmapsto 2^{-\alpha_0}\, \exp \Big( - \frac{(x-\tau)^2}{2^{\alpha_0}} + i\, \omega \, x\Big)
$$
involved in Gabor time-frequency analysis. Then $\psi_{\alpha_0,\tau,\omega}$ is the uniform limit
on any compact subset of $\boldsymbol {\mathcal R}$ of the sequence of functions
$$
\Big(
\sum_{\bfnu \prec \bfN = (N_1,N_2)\in (\N^*)^2} \binom{\bfN}{\bfnu}
\Big( \frac{1 - (\tau,\omega)}{2}\Big)^{\bfnu}
\Big( \frac{1 + (\tau,\omega)}{2}\Big)^{\bfN - \bfnu}\, \psi_{\alpha_0, 1- 2 \nu_1/N_1, 1-2 \nu_2/N_2}\Big)_{\bfN = (N_1,N_2)}
$$
when $\min (N_1,N_2) \rightarrow +\infty$, where notations are those introduced in the preamble of \textsection \ref{sect1}.
\end{proposition}

\begin{proof} It follows from Lemma \ref{SECT5-lem2} with $\tau_1=0$, $\tau_2=\tau$ and $\alpha = \alpha_0$, together with Theorem \eqref{sect2-thm1}.
\end{proof}

\begin{proposition}\label{SECT5-prop2} Let $(t,x)\in \boldsymbol{\mathcal R} \longmapsto \psi_{\tau,\alpha,\omega} (t,x)$ be the evolution under $\mathscr S_0$ of the translated, scaled and modulated atom
$$
x \longmapsto 2^{-\alpha/2}\, \Psi\Big(\frac{x-\tau}{2^\alpha}\Big)\, e^{i\omega x},
$$
where $\Psi$ is a wavelet which spectrum is a compactly supported integrable signal (as the Shannon's or Meyer's wavelet).
Then $\psi_{\tau,\alpha,\omega}$ is the uniform limit
on any compact subset of $\boldsymbol {\mathcal R}$ of the sequence of functions
$$
\Big(
\sum_{\bfnu \prec \bfN = (N_1,N_2,N_3)\in (\N^*)^3} \binom{\bfN}{\bfnu}
\Big( \frac{1 - (\tau,\alpha,\omega)}{2}\Big)^{\bfnu}
\Big( \frac{1 + (\tau,\alpha,\omega)}{2}\Big)^{\bfN - \bfnu}\, \psi_{{\bf 1} - 2 \bfnu/\bfN}\Big)_{\bfN = (N_1,N_2,N_3)}
$$
when $\min (N_1,N_2,N_3) \rightarrow +\infty$, where notations are those introduced in the preamble of \textsection \ref{sect1}.
\end{proposition}

\begin{proof} One applies Lemma \ref{SECT5-lem4}, then Theorem \ref{sect2-thm1}.
\end{proof}

The next two propositions illustrate Theorem \ref{sect3-thm1} when $d=1$.

\begin{proposition}\label{SECT5-prop3} Let $T \in ]0,1[$ and $\mathcal T_1, \mathcal T_2$, $\mathcal A$
be continuous maps from $[-T,T]$ respectively to $[0,+\infty[$, $\R$, $]-\infty,+\infty]$ such that
$\mathcal T_1 (\pm T) =0$ and $\mathcal A(\pm T) = +\infty$. Let
$$
(t,x) \longmapsto \psi_\omega^{\boldsymbol {\mathcal T}, \mathcal A} (t,x)
$$
be the evolution under $\mathscr S_0$ of the continuous signal
$$
x \longmapsto \begin{cases} \exp (i\, \omega\, x)\ \text{when}\ x\notin [-T,T] \\ \\
\exp \Big( i \, \Big( \omega + \displaystyle{\frac{\mathcal T_1(\omega)}{2}}\, x\Big)\, x -
\frac{1}{2^{\mathcal A (\omega)}} \Big( x - \mathcal T_2 (\omega)\Big)^2\Big) \Big)\  \text{when}\ x\in [-T,T].
\end{cases}
$$
Then $\omega \longmapsto \psi_\omega^{\boldsymbol {\mathcal T}, \mathcal A} (t,x)$ is $T$-predictable with respect to the parameters $(t,x) \in \boldsymbol {\mathcal R}$, hence satisfies Theorem \ref{sect3-thm1} with $T\in ]0,1[$.
\end{proposition}

\begin{proof}
The continuity of $(\omega,t,x) \longmapsto \psi_\omega^{\boldsymbol {\mathcal T}, \mathcal A} (t,x)$ follows from Lemmas \ref{SECT5-lem1} and \ref{SECT5-lem2}. The result then follows from the fact that
$$
\psi_\omega^{\boldsymbol {\mathcal T}, \mathcal A} (t,x) = \exp (i (\omega\, x - \omega^2\, t))
$$
for any $\omega\leq -T$ or $\geq T$ and any $(t,x) \in \boldsymbol {\mathcal R}$, that is as a function of $\omega$ the restriction to the real line of an entire function depending continuously in
$(t,x) \in \boldsymbol {\mathcal R}$.
\end{proof}

\begin{proposition}\label{SECT5-prop4}
Let $(t,x) \longmapsto \psi_{\omega,\pm} (t,x)$ be the evolution under $\mathscr S_0$ of the continuous signal $x \longmapsto \exp (i\, |\omega|\, | x|)$ (with switch of frequency sign at the origin). Then $\psi_{\omega,\pm}$ is $0$-predictable with respect to the parameters
$(t,x) \in \boldsymbol {\mathcal R}$, hence satisfies Theorem \ref{sect3-thm1} with $T=0$.
\end{proposition}

\begin{proof} It follows from Lemma \ref{SECT5-lem3} and from the fact that
$$
Z \longmapsto \int_0^ Z \exp \Big (i \frac{z^2}{4t}\Big)\, dz
$$
is an entire function of $Z$ depending continuously of $t>0$ that
$\omega \mapsto \psi_{\omega,\pm}$ is the restriction to the real line of an entire function of
$|z|$ which depends continuously of the parameters $(t,x)\in \boldsymbol {\mathcal R}$.
\end{proof}


\section*{Declarations and statements}

{\bf Data availability}. The research in this paper does not imply use of data.

{\bf Conflict of interest}. The authors declare that there is no conflict of interest.

{\bf Acknowledgments}.
F. Colombo and I. Sabadini are supported by MUR grant Dipartimento di Eccellenza 2023-2027.


\begin{thebibliography}{99}


\bibitem{aav} Y. Aharonov, D. Albert, L. Vaidman, {\em How the result of a measurement of a component of the spin of a spin-1/2 particle can turn out to be 100}, Phys. Rev. Lett., {\bf 60} (1988), 1351-1354.

\bibitem{abook} Y. Aharonov, D. Rohrlich, {\em Quantum Paradoxes: Quantum Theory for the Perplexed}, Wiley-VCH Verlag, Weinheim, 2005.


\bibitem{ABCS22} Y. Aharonov, J. Berndt, F. Colombo, P. Schlosser,
{\it A unified approach to Schr\"odinger evolution of superoscillations and supershifts},
  J. Evol. Equ. 22 (2022), no. 1, Paper No. 26, 31 pp.

\bibitem{JDE}
Y. Aharonov, J. Behrndt, F. Colombo, P. Schlosser,
{\em Green's Function for the Schr\"odinger Equation with
 a Generalized Point Interaction and Stability of Superoscillations},
  J. Differential Equations, {\bf 277} (2021), 153--190.

\bibitem{ACSSST21}
Y. Aharonov, F. Colombo, I. Sabadini, T. Shushi, D.C. Struppa, J. Tollaksen, {\it A new method to generate superoscillating
functions and supershifts}, Proc. R. Soc. A. \textbf{477} (2249), Paper No. 20210020, 12 pp (2021).

\bibitem{ACJSSST22} Y. Aharonov,  F. Colombo, A.N. Jordan, I. Sabadini, T. Shushi, D.C. Struppa, J. Tollaksen, \textit{On superoscillations and supershifts in several variables},
 Quantum Stud. Math. Found. 9 (2022), no. 4, 417-433.

\bibitem{ACSST17A} Y. Aharonov,  F. Colombo,  I. Sabadini, D.C. Struppa, J. Tollaksen,
{\em The mathematics of superoscillations}, Mem. Amer. Math. Soc. 247 (2017), no. 1174, v+107 pp.


\bibitem{Bern12} S. Bernstein,
\textit{D\'emonstration du th\'eor\`eme de Weierstrass, fond\'ee sur le calcul des probabilit\'es},
Commun. Soc. Math. Kharkow {\textbf 2}, 13 (1912-13), 1-2.

\bibitem{Bern18} S. Bernstein,
\textit{Quelques remarques sur l'interpolation},
Math. Annalen \textbf{79} (1918), pp. 1-12.

\bibitem{Bern35} S. Bernstein,
\textit{Sur la convergence de certaines suites de polyn\^omes}, J. Math. Pures Appl. {\textbf 15}, no. 9 (1935), pp. 345-358.

\bibitem{BOREL} J. Behrndt, F. Colombo, P. Schlosser, D.C. Struppa, {\em Integral representation of superoscillations via complex Borel measures and their convergence}, Trans. Amer. Math. Soc. 376 (2023), no. 9, 6315-6340.

\bibitem{b4} M. V. Berry, S. Popescu,
{\em Evolution of quantum superoscillations, and optical superresolution
without evanescent waves},
{ J. Phys. A}, {\bf 39} (2006), 6965--6977.

\bibitem{Be19} M. Berry et al, \textit{Roadmap on superoscillations}, 2019, Journal of Optics 21 053002.

\bibitem{b5} M.V. Berry, P. Shukla, {\em Pointer supershifts and superoscillations in weak measurements}, {\em J. Phys A}, {\bf 45} (2012), 015301.

\bibitem {CPSS23} F. Colombo, S. Pinton, I. Sabadini, D.C. Struppa, \textit{The General Theory of Superoscillations and Supershifts in Several Variables}, J. Fourier Anal. Appl. 29 (2023), no. 6, Paper No. 66, 24 pp.
29:66.

\bibitem{CSSY22} F. Colombo, I. Sabadini, D.C. Struppa, A. Yger,
{\em Superoscillating sequences and supershifts for families of
generalized functions}, Complex Analysis and Operator Theory 16 (2022), article 34.


\bibitem{CSSY21} F. Colombo, I. Sabadini, D.C. Struppa, A. Yger,
{\it Gauss sums, superoscillations and the Talbot carpet}, J. Math. Pures Appl. \textbf{147} (2021) pp. 163--178.

\bibitem{SPIN} F. Colombo, E. Pozzi, I. Sabadini, B. D. Wick,
{\em Evolution of superoscillations for spinning particles},
Proc. Amer. Math. Soc. Ser. B 10 (2023), 129-143.

 \bibitem{cssy23}
F. Colombo, I. Sabadini, D. C. Struppa, A. Yger,
{\em Analyticity and supershift with regular sampling},
Preprint arXiv:2310.11528.



\bibitem{IRREGULAR-SAMP}
F. Colombo, I. Sabadini, D. C. Struppa, A. Yger,
{\em Analyticity and supershift with irregular sampling}, To appear in Complex Analysis and its Synergies, preprint available on
arXiv:2312.05089.

\bibitem{Davis75} P. J. Davis, \textit{Interpolation and approximation},
Dover Publications, Inc., New York, 1975.

\bibitem {DMBH16}
N.Q. Dieu, P.V. Manh, P.H. Bang, L.T. Hung,
\textit{Vitali's theorem without uniform boundedness}, Publ. Mat. \textbf{60} (2016), pp. 311-334.

\bibitem{ErdV80} P. Erd\"os, P. V\'ertesi,
\textit{On the almost everywhere divergence of Lagrange interpolatory polynomials for arbitrary system of nodes},  Acta Math. Acad. Sci. Hungar. 36 (1980), no. 1-2, 71-89.


\bibitem{OPTICS} F.M. Huang, Y. Chen, F.J.G. De Abajo, N.I. Zheludev,
{\em Optical super-resolution through super-oscillations},
J. Optics A: Pure Appl.Optics, 9(9) (2007) p.285.



\bibitem{SPGR} T. Karmakar, S.A. Wadood, A.N. Jordan, A.N. Vamivakas,
{\em Experimental realization of supergrowing fields}, arXiv:2309.00016.


\bibitem{KM1} T. Karmakar, A. Chakraborty, A.N. Vamivakas, A.N. Jordan,
{\em Supergrowth and sub-wavelength object imaging}, Opt. Express, 31(22) (2023) pp.37174-37185.

\bibitem{Kanto31} L. Kantorovitch,
\textit{Sur la convergence de la suite des polyn\^omes de S. Bernstein en dehors de l'intervalle
fondamental}, Bull. Acad. Sci. URSS (1931), pp. 1103-1115.

\bibitem{Phil03}
G. M. Phillips, \textit{Interpolation and Approximation by Polynomials}, CMS Books in Mathematics, Springer Science + Business Media New York, 2003.

\bibitem{Kow16} E. Kowalski, \textit{Bernstein polynomials and Brownian motion}, Amer. Math. Monthly 113 (2006), no. 10, 865–886.

\bibitem{Lor53} G.G. Lorentz,
\textit{Bernstein polynomials}, Toronto, 1953 (second edition : Chelsea Publishing Company, New York, 1986).

\bibitem{QS20} A.N. Jordan, {\em Superresolution using supergrowth and intensity contrast imaging},
Quantum Stud. Math. Found., 7(3) (2020) 285-292.



\bibitem{Neum1862}
K. Neumann, \textit{Ub\"er die Entwickelung Einer Funktion Nach den Kugelfunktionen}, Journal f\"ur Mathematik, 1862.


\bibitem{Pozzi}
E. Pozzi,  B. D. Wick, {\em Persistence of superoscillations under the Schrödinger equation},
 Evol. Equ. Control Theory, 11 (2022), no. 3, 869-894.


\bibitem{Rev2000} M. Revers,
\textit{The divergence of Lagrange interpolation for $|x|^\alpha$ at equidistant nodes},
 J. Approx. Theory, 103 (2000), no. 2, 269-280.

\bibitem{peter}
 P. Schlosser,  {\em Time evolution of superoscillations for the Schr\"odinger equation on $\mathbb{R}\setminus \{0\}$},
  Quantum Stud. Math. Found. 9 (2022), no. 3, 343–366.


\bibitem{sodakemp}
B. Šoda,  A. Kempf, {\em  Efficient method to create superoscillations with generic target behavior},
Quantum Stud. Math. Found., {\bf 7} (2020), no. 3, 347--353.


\end{thebibliography}
\end{document}